\documentclass[12pt,a4paper]{article}
\usepackage{latexsym,amsfonts,amssymb,color,amsmath,a4wide,times,amsthm}
\usepackage[dvips]{epsfig}
\title%[The evolution of random planar graphs]
{Two critical periods in the evolution\\
of random planar graphs}
\author{Mihyun Kang\\
Institut f\"ur Mathematik, MA 6-2,
Technische Universit\"at Berlin\\
Stra{\ss}e des 17. Juni 136,
10623 Berlin,
Germany\\
kang@math.tu-berlin.de\\
and\\
Tomasz {\L}uczak\\
Department of Discrete Mathematics, Adam Mickiewicz University\\
61-614 Pozna\'n, Poland\\
tomasz@amu.edu.pl}

\date{8 November 2010}

\newtheorem{theorem}{Theorem}
\newtheorem{lemma}{Lemma}

\newcommand{\pr}{\mathbb{P}}

\newcommand{\aas}{a.a.s.\ }
\newcommand{\eps}{\epsilon}

\newcommand{\Q}{\mathcal Q}
\newcommand{\q}{\mathbf q}

\newcommand{\stacksign}[2]{{\stackrel{\textrm{\scriptsize #1}}{#2}}}

\newcommand{\pl}{\mathrm {pl}}
\newcommand{\ex}{\mathrm {ex}}
\newcommand{\core}{\mathrm{core}}
\renewcommand{\ker}{\mathrm{ker}}
\newcommand{\dd}{\mathrm{df}}

\newcommand{\Pcr}{\mathrm {cr}}
\newcommand{\Gcr}{\bar{\mathrm {cr}}}
\newcommand{\Gex}{\bar{\ex}}

\begin{document}
\maketitle

\begin{abstract}
Let $P(n,M)$ be a graph chosen uniformly at random from the family of
all labeled planar graphs with $n$ vertices and $M$ edges.
In the paper we study the component structure of $P(n,M)$.
Combining counting arguments with analytic techniques,
we show that
there are two critical periods in the evolution of
$P(n,M)$. The first one, of width $\Theta(n^{2/3})$, is
analogous to the phase transition observed in
the standard random graph models and
takes place for  $M=n/2+O(n^{2/3})$,
when the largest complex component is formed.
Then, for $M=n+O(n^{3/5})$,
when the complex components cover nearly all vertices, the second critical
period of width $n^{3/5}$ occurs. Starting from that moment increasing
of $M$ mostly affects the density of the complex components,
not  its size.
\end{abstract}

%Keywords: planar graphs; random graphs; asymptotic enumeration; phase transition
% 2010 Mathematics Subject Classification: 05C10 (planar graphs); 05C80 (random graphs);  05C30 (enumeration); 05A16 (asymptotic enumeration)

\section{Introduction}
Since the seminal work of Tutte~\cite{censusmaps63} maps
and  graphs on 2-dimensional surfaces have become widely
studied combinatorial objects in discrete mathematics.
The enumerative and structural problems around maps,
i.e.\ \emph{embedded} graphs on a surface, are relatively well settled.
Starting from the number of planar maps computed
by Tutte~\cite{Tutte62,censusmaps63},
the number of rooted maps on surfaces was found by Bender,
Canfield, and Richmond~\cite{BC91} and other classes
of maps were extensively enumerated since then.
Such enumeration results have been used to study typical properties
of random maps on surfaces, e.g.\ the size of the largest components
by Banderier {\it et al.}~\cite{Airy}.

Another important aspect of maps is that they allow nice bijections
to the so-called {\em well-labeled} trees.
The bijection between planar maps and the well-labeled trees
was first studied amongst others by Schaeffer~\cite{Schaeffer-Thesis},
which was extended by Bouttier, Di Francesco, and Guitter ~\cite{BouttierFG}
to maps on an orientable surface with positive genus.
These bijections are the corner stone of profound results
on the topological structure of scaling limits of random maps
by Chassaing and Schaeffer~\cite{ChassaingSchaeffer},
Le Gall~\cite{LeGall},  and  Schramm~\cite{Schramm}.

On the other hand, analogous problems related graphs that are
\emph{embeddable} on a surface are still wide open.
The enumerative properties of random planar graphs have attracted much
attention since the work of McDiarmid, Steger, and Welsh~\cite{MSW}
who studied random  labeled planar graphs with a given number of vertices.

Let $\pl(n)$ be the number of labeled planar graphs on $n$ vertices.
McDiarmid, Steger, and Welsh~\cite{MSW} showed amongst other results
that the quantity $(\pl(n)/n!)^{1/n}$ converges to a limit $\gamma$
as $n \to \infty$, which is called the {growth constant}.
An upper bound $\gamma\le 37.3$, based on the triangulations
and probabilistic methods, was obtained
by Osthus, Pr\"omel, and Taraz~\cite{OPT03}.
A lower bound $\gamma\ge 26.1$ was given by Bender, Gao,
and Wormald~\cite{BenderGaoWormald} who studied the number
of labeled 2-connected planar graphs through the singularity
analysis of generating functions arising from the decomposition
of graphs along connectivity.
Using a similar method Gim{\'e}nez and Noy~\cite{GN}
 proved that $\pl(n)\sim c\ n^{-7/2}\ \gamma^n\  n!$,
where  $c,\gamma >0$ are explicitly computable constants
(e.g.\ $\gamma\sim 27.2$).
As for the number $s_g(n)$ of graphs embeddable
on a surface with positive genus,
McDiarmid~\cite{McDiarmid} showed that its growth constant
 is the same as that of planar graphs while
Chapuy {\it et al.}~\cite{CFGMN} found $s_g(n)$
for all orientable surfaces of positive genus $g$ proving that
$s_g(n) \sim \alpha_g\  n^{5(g-1)/2-1}\  \gamma^n\  n!$,
where $\alpha_g>0$  and $\gamma$
is the growth constant of the labeled planar graphs.
McDiarmid and Reed~\cite{MR} studied various typical properties
of random graphs on surfaces, e.g.\ subgraph containment and maximum degree.

Frieze~\cite{Frieze} asked about the asymptotic  behavior of
the number  $\pl(n,M)$ of labeled planar graphs on $n$
vertices  with $M$ edges. Note first that if $M\le a n$ for some
$a<1/2$, then a typical graph with $n$ vertices and $M$ edges is
planar, i.e.
$\pl(n,M)=(1+o(1))\binom {\binom n2} M$ (cf.\ \cite{LPW} or  \cite{JLR})
Gerke {\it et al.}~\cite{GMSW} proved the existence of
its growth constant, in the sense that for $\pl(n,an)$ with $0\le a\le 3$,
the quantity $(\pl(n,an)/n!)^{1/n}$ converges to a limit $\gamma_a$
as $n \to \infty$. The asymptotic formula for $\pl(n,an)$ was found
by Gim{\'e}nez and Noy~\cite{GN} who showed that for $1<a<3$,
there are analytic constants $c_a, \gamma_a>0$ such that
$\pl(n,an)\sim c_a\ n^{-4}\ \gamma_a^n\  n!$. In this paper
we deal with the case when $a\in [1/2,1)$. Note that for such $a$
$$ \pl(n,M)\le \binom {\binom n2} M= n^{(1+o(1))M},$$
i.e., $\gamma_a=0$, and a `naive' generating function approach
does not lead to the asymptotic formula for $\pl(n,M)$.

We use $\pl(n,M)$ to study the asymptotic behavior of the
uniform random planar graph $P(n,M)$ by which we mean
a graph chosen uniformly
at random among all labeled planar graphs with $n$ vertices and $M$ edges;
thus, each of these graphs occurs as $P(n,M)$ with probability $1/\pl(n,M)$.
From  the results in~\cite{GMSW, GN} it follows
that if $M/n$ is bounded away from both
$1$ and $3$, then
$P(n,M)$ has a well ordered structure, for instance,
it has a large component of size $n-O(1)$, and all
planar graphs of finite size  appear as its subgraphs.
Thus, it corresponds to late stages of the evolution
of the standard uniform random graph $G(n,M)$, the graph chosen
uniformly at random among all graphs with $n$ vertices and $M$ edges.
Our goal is to study  the typical size and structure of components
in $P(n,M)$ in a more interesting range, when $M\le n$.
It turns out that, somewhat surprisingly,
$P(n,M)$ exhibits two critical ranges, which occur at $M=n/2+O(n^{2/3})$
and $M=n+O(n^{3/5})$.

The first critical period corresponds to a
phase transition phenomenon observed in the plethora of
different random graph models. Let us recall
some results on one of the most widely used
random graph model  $G(n,M)$. It follows from the papers of
Erd\H{o}s and R\'enyi \cite{ER60},
Bollob\'as~\cite{Bollobas84}, {\L}uczak~\cite{L},
{\L}uczak, Pittel, and Wierman~\cite{LPW},
Janson {\it et al.}~\cite{JKLP}, and Janson~\cite{J}
(see also Janson, {\L}uczak, and Ruci\'nski~\cite{JLR})
that the giant component (i.e.\ the unique largest component)
suddenly emerges at $M=n/2+O(n^{2/3})$, and nowadays this
spectacular phenomenon is well studied and understood.
If  $M=n/2+s$ and
 $-n\ll s\ll -n^{2/3}$, then, \aas (i.e.\ with probability
tending to 1 as $n$ approaches $\infty$)  $G(n,M)$ consists of
isolated trees and unicyclic components, and
the largest component is a tree
of size $(1+o(1))\frac{n^2}{2s^2} \log \frac{|s|^3}{n^2}$.
On the other hand, if  $n^{2/3}\ll s\ll n$,
then \aas $G(n,M)$ contains exactly one
component with more edges than vertices  of size $(4+o(1))s$,
while all other components are of size $o(n^{2/3})$.
Furthermore, if $s\gg n^{2/3}$, then  \aas $G(n,M)$ contains
a topological copy of $K_{3,3}$ and thus it is not planar,
while, as we have mentioned,  for $s\ll -n^{2/3}$, \aas $G(n,M)$
consists of isolated trees and unicyclic components, so it is
clearly  planar.

Another random structure relevant to the behavior of $P(n,M)$
is the uniform random forest $F(n,M)$ (i.e.\ a forest chosen uniformly
at random among all labeled forest with $n$ vertices and $M$ edges).
{\L}uczak and Pittel~\cite{LP} found that although
the giant component in $F(n,M)$ emerges at $M=n/2+O(n^{2/3})$,
as for $G(n,M)$, the critical behavior of these two
models are somewhat different. Let $M=n/2+s$.
If $s\ll -n^{2/3}$, then the structure of both
$F(n,M)$ and $G(n,M)$ are similar; in particular, the
size of the largest tree in $F(n,M)$ is \aas
$(1+o(1))\frac{n^2}{2s^2} \log \frac{|s|^3}{n^2}$.
However in the supercritical phase, when  $s\gg n^{2/3}$,
the giant tree of $F(n,M)$
is \aas of size $(2+o(1))s$, which is roughly half of the size
of the largest component
of $G(n,M)$,
while  the second largest tree of $F(n,M)$ is of size
$\Theta(n^{2/3})$ which does not depend much on $s$ provided
$s\ll n$, i.e.\ it is by far larger
than the second largest component of $G(n,M)$ for $m=n/2+s$,
which is of size $\Theta(n^2/s^2\log (s^3/n^2))$.

In the paper we show that as far as $M=n/2+s$, where  $s\gg n^{2/3}$
and $s/n$ is bounded away from $1/2$, the behavior of $P(n,M)$
is similar to that of $F(n,M)$. Namely, \aas the size of the largest
complex component is of the order $(2+o(1))s$, while the second
largest component has $\Theta(n^{2/3})$ vertices.
However, unlike in the case of $F(n,M)$ for which $M\le n-1$,
for $P(n,M)$ we may have $M= 3n-6$, so the rate of growth
of the size of complex  components  must change at some point.
We prove that it occurs when $M=n+O(n^{3/5})$, more precisely,
if $M=n+t$ and $t\ll -n^{3/5}$, then the complex components
of $P(n,M)$ have \aas $n-(2+o(1))|t|$ vertices altogether,
while for $n^{3/5}\ll t\ll n^{2/3}$
they contain $n-(\alpha +o(1))(n/t)^{3/2}$  vertices for some computable
constant $\alpha>0$. Let us mention that the condition
$t\ll n^{2/3}$ is a  result of the proof method we have used
and most likely can be replaced by $t\ll n$.
Furthermore, our method can say quite a lot about structure
of the largest component very much in the spirit of {\L}uczak~\cite{cycle}.

The rest of the paper is organized as follows.
In the next section we describe the main idea of our argument.
Then, in Section~\ref{sec:cubicplanar}
we present the first analytic ingredient of the proof: counting
specially weighted cubic planar multigraphs using generating functions.
Here we also describe how to use this result to
bound the number of planar multigraphs
with minimum degree three. Then, in the next section, we
estimate the number of planar graph with $k$ vertices
and $k+\ell$ edges in which each component has more
edges than vertices. Finally, in the main chapter
of this paper, we use a direct  counting  to
study the number $\pl(n,M)$ and the asymptotic properties
of $P(n,M)$ for different values of $M$.

\section{Idea of the proof}

As we have already mentioned most of results concerning the
asymptotic behavior of $\pl(n,M)$ is based on the generating
function method. Thus, for $a\in (0,3)$, one can study the function
$$f_a(x)=\sum_{n}\frac{\pl(n,an)}{n!}x^n\,$$
and deduce the asymptotic behavior of $\pl(n,an)$ from
the behavior of $f_a(x)$ near its singularities.
Note however that since
$$\pl(n,an)\le \binom {\binom n2}{an}\le n^{(a+o(1))n}\,,$$
for $a<1$, the coefficients in the expansion of $f_a(x)$ tend
to zero too fast to be handled by standard methods of
generating function analysis. On the other hand, the condition
that a graph is planar is very hard to grasp by purely
combinatorial means. Thus, in the paper we use a combination of
analytic and combinatorial tools. From a planar
graph we extract its kernel,
which is its only part responsible for the planarity and is
dense enough to be treated by generating functions method.
Then we use a technically  involved but rather natural counting
argument to find asymptotic properties of $P(n,M)$.

In order to make the above description precise,
we introduce some definition. The {\em excess} $\ex(G)$ of a
graph $G$ is the difference between the number of its
edges and the number of its vertices.
We call components of a graph with positive excess, i.e.\
those which have at least two cycles, {\em complex} components,
and we say that the graph is complex if all its
components are complex.
The core of a graph $G$, denoted $\core(G)$,  is
the maximal subgraph of $G$ with the minimum degree two.
The kernel of $G$, denoted by $\ker(G)$,
is obtained from the core by removing all isolated cycles
and replacing each path
whose internal vertices are all of degree two by an edge.
Note that for a graph $G$ the kernel is a multigraph which
can have  multiple edges and loops.
However, the excess of both $\core(G)$ and $\ker(G)$ are the same.
Note also that $\ker(G)$ has clearly the minimum degree three.
If $\ker(G)$  is cubic, we say that $G$ is {\em clean}.
We define the {\em deficiency} $\dd(G)$ of $G$ as
the sum of degrees of vertices of $\ker(G)$
minus three times the number of vertices of $\ker(G)$.
Therefore, a graph is clean if and only if its deficiency is zero.

Our argument is based on a simple observation that
$G$ is planar if and only if $\ker(G)$ is planar.
We use analytic methods to count the number
of possible candidates for the kernel, and then follow
purely combinatorial argument.
To this end, in Section~\ref{sec:cubicplanar}
we apply the singularity analysis to extract
the asymptotic number of possible
clean kernels of the planar graphs (Theorem~\ref{thm:cubic-asymptotics}).
The proof is similar to that used in the case
of counting cubic planar graphs presented
in Bodirsky {\it et al.}~\cite{BKLM}, but in order to
find the number of all planar graphs,
we need to count cubic planar multigraphs with some
special weight function dependent on the number of loops
and multiple edges of a multigraph.
Then, we use a simple combinatorial idea to generalize our estimates
to the number of kernels with non-zero deficiency.

In the remaining part of our argument we follow
an idea from {\L}uczak~\cite{cycle}
(see also Janson, {\L}uczak, and Ruci\'nski~\cite{JLR})
and construct  complex  planar graphs $G$  from their kernels.
Thus, we first choose the kernel, then put on its edges
vertices of degree two obtaining the core of the graph,
and add to it a forest rooted on vertices of the core.
This procedure lead to the estimate given in Theorem~\ref{thm:complex}.
We also remark that such a typical complex planar graph
on $k$ vertices consists of large component and,
perhaps, some small components of combined excess $O(1)$.

Finally, in the main part of the paper we count
the number of planar graphs by splitting
it into a complex planar graph, the number of
which we have just found, and the part which consists
of isolated trees and unicyclic components,
whose number is well known (see Britikov~\cite{Britikov}).
We compute $\pl(n,M)$ for different values of $M=M(n)\le n+o(n)$.
At the same time, we get information on the typical structure of $P(n,M)$
such as the size of the largest component, its excess, and core.

\section{Cubic planar graphs}\label{sec:cubicplanar}

In this section we study the family of cubic planar weighted
multigraphs, which plays a crucial role in studying
the kernel of a random planar graph.
We then consider the family of `supercubic' planar graphs
that are planar weighted multigraphs with minimum degree
at least three and a positive deficiency, which is indeed
the set of possible kernels of complex planar graphs.

\subsection{Cubic planar weighted multigraphs}\label{subsec:cubicplanar}
In this section we count the number of all labeled cubic planar weighted
multigraphs, where each multigraph with $f_1$ loops,
$f_2$ double edges, and $f_3$ triple edges gets
weight $2^{-f_1-f_2}6^{-f_3}$.
For $k=0,1$, let $g^{(k)}_n$ be the number of  all labeled $k$-vertex
connected cubic planar weighted multigraphs on $n$ vertices
and $G^{(k)}(x)$ be the corresponding exponential generating
function defined by
\begin{equation*}
G^{(k)}(x)  := \sum_{n \ge 0}\frac{g^{(k)}_{n}}{n!}\, x^n.
\end{equation*}
Note that $g^{(k)}_n=0$ for odd $n$ and also for $n=0$ except that we set
$g_0^{(0)}=1$ by convention. It is well-known (e.g.\ \cite{GN,HP}) that
\begin{equation*}
G^{(0)}(x) =  \exp(G^{(1)}(x)).
\end{equation*}

The function $G^{(1)}(x)$ is defined by the following system of equations:
\begin{equation}\label{eq:genfunctions}
\begin{aligned}
3x\, \frac{dG^{(1)}(x)}{dx} &=   D(x)+S(x)+(P(x)-x^2/4+x^2/12)+H(x)\\
&\hspace{3truecm}+(B(x)-x^2/4+x^2/8)\\
 &=    D(x)+C(x)-7x^2/24\,\\
B(x) &=  x^2(D(x)+C(x))/2+x^2/4\\
C(x) &=  S(x)+P(x)+H(x)+B(x)\\
D(x) &=   B(x)^2/x^2-x^2/16\\
S(x) &=  C(x)^2 - C(x)S(x)\\
P(x) &=   x^2 C(x)+ x^2 C(x)^2/2+x^2/4\\
2(C(x)+1)H(x) &=  {u(1-2u) - u(1-u)^3} \\
x^2 (C(x)+1)^3 &=  u(1-u)^3.
\end{aligned}
\end{equation}
This system of equations is obtained by following the lines of
Sections 3--6 in Bodirsky {\it et al.}~\cite{BKLM}, where cubic planar \emph{simple}
graphs were studied, so below we just outline the argument.

The starting idea is that given a connected cubic planar weighted
multigraph $G$, we select an arbitrary edge $e$ in $G$
and orient the edge $e$, to obtain a \emph{rooted} counterpart $\hat G$.
More precisely, the {rooted cubic graph} $\hat G=(V, E, st)$
obtained from a connected cubic multigraph $G=(V, E)$
consists of $G=(V, E)$ and an ordered pair of adjacent vertices $s$ and $t$.
The oriented edge $st$ is called the {root} of $\hat G$.
Denote by $G^-$  a graph obtained from $\hat G$
by deleting the root of $\hat G$. We have the following lemma
analogous to Lemma 1 in \cite{BKLM}.

\begin{lemma}\label{lem:structure}
A rooted cubic graph $\hat G=(V,E,st)$ has exactly
one of the following types.
\begin{itemize}
\item $b$: the root is a self-loop.
\item $d$: $G^{-}$ is disconnected.
\item $s$: $G^{-}$ is connected, but there is a cut edge
in $G^{-}$ that separates $s$ and $t$.
\item $p$: $G^{-}$ is connected, there is no cut edge in $G^{-}$
 separating $s$ and $t$, and either $st$ is an edge of $G^{-}$,
 or $G \setminus \{s,t\}$ is disconnected.
\item $h$: $G^{-}$ is connected, there is no cut-edge in $G^{-}$
 separating $s$ and $t$, $G$ is simple, and $G \setminus \{s,t\}$
 is connected.
\end{itemize}
\end{lemma}
\begin{figure}[h]
\begin{center}
\begin{picture}(0,0)%
\includegraphics{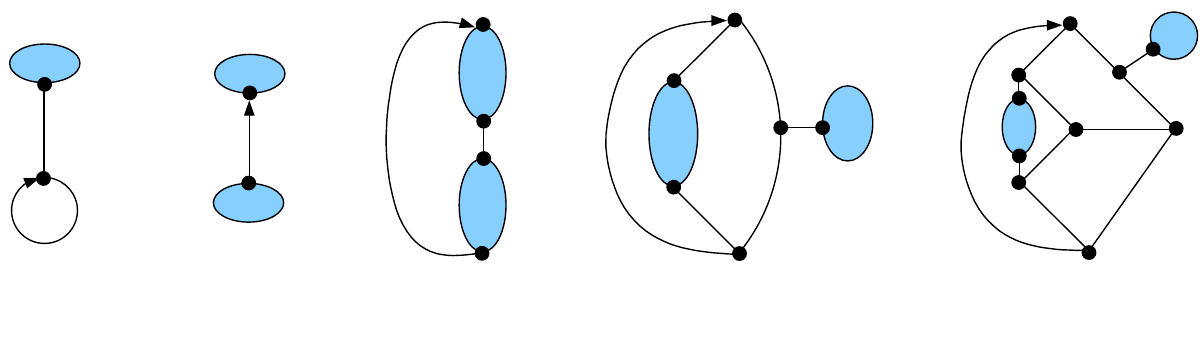}%
\end{picture}%
\setlength{\unitlength}{3522sp}%
\begingroup\makeatletter\ifx\SetFigFont\undefined%
\gdef\SetFigFont#1#2#3#4#5{%
  \reset@font\fontsize{#1}{#2pt}%
  \fontfamily{#3}\fontseries{#4}\fontshape{#5}%
  \selectfont}%
\fi\endgroup%
\begin{picture}(6450,1820)(-14,-1115)
\put(1393,-249){\makebox(0,0)[lb]{\smash{{\SetFigFont{8}{9.6}{\familydefault}{\mddefault}{\updefault}{\color[rgb]{0,0,0}$s$}%
}}}}
\put(1398,105){\makebox(0,0)[lb]{\smash{{\SetFigFont{8}{9.6}{\familydefault}{\mddefault}{\updefault}{\color[rgb]{0,0,0}$t$}%
}}}}
\put(1081,-1051){\makebox(0,0)[lb]{\smash{{\SetFigFont{8}{9.6}{\rmdefault}{\mddefault}{\updefault}{\color[rgb]{0,0,0}$d$-graph}%
}}}}
\put(301,-215){\makebox(0,0)[lb]{\smash{{\SetFigFont{8}{9.6}{\familydefault}{\mddefault}{\updefault}{\color[rgb]{0,0,0}$s=t$}%
}}}}
\put(2678,554){\makebox(0,0)[lb]{\smash{{\SetFigFont{8}{9.6}{\familydefault}{\mddefault}{\updefault}{\color[rgb]{0,0,0}$t$}%
}}}}
\put(2674,-698){\makebox(0,0)[lb]{\smash{{\SetFigFont{8}{9.6}{\familydefault}{\mddefault}{\updefault}{\color[rgb]{0,0,0}$s$}%
}}}}
\put(2251,-1051){\makebox(0,0)[lb]{\smash{{\SetFigFont{8}{9.6}{\rmdefault}{\mddefault}{\updefault}{\color[rgb]{0,0,0}$s$-graph}%
}}}}
\put(4050,-703){\makebox(0,0)[lb]{\smash{{\SetFigFont{8}{9.6}{\familydefault}{\mddefault}{\updefault}{\color[rgb]{0,0,0}$s$}%
}}}}
\put(4044,582){\makebox(0,0)[lb]{\smash{{\SetFigFont{8}{9.6}{\familydefault}{\mddefault}{\updefault}{\color[rgb]{0,0,0}$t$}%
}}}}
\put(5934,-703){\makebox(0,0)[lb]{\smash{{\SetFigFont{8}{9.6}{\familydefault}{\mddefault}{\updefault}{\color[rgb]{0,0,0}$s$}%
}}}}
\put(5826,563){\makebox(0,0)[lb]{\smash{{\SetFigFont{8}{9.6}{\familydefault}{\mddefault}{\updefault}{\color[rgb]{0,0,0}$t$}%
}}}}
\put(  1,-1051){\makebox(0,0)[lb]{\smash{{\SetFigFont{8}{9.6}{\rmdefault}{\mddefault}{\updefault}{\color[rgb]{0,0,0}$b$-graph}%
}}}}
\put(3781,-1051){\makebox(0,0)[lb]{\smash{{\SetFigFont{8}{9.6}{\rmdefault}{\mddefault}{\updefault}{\color[rgb]{0,0,0}$p$-graph}%
}}}}
\put(5626,-1051){\makebox(0,0)[lb]{\smash{{\SetFigFont{8}{9.6}{\familydefault}{\mddefault}{\updefault}{\color[rgb]{0,0,0}$h$-graph}%
}}}}
\end{picture}%
\end{center}
\caption{The five types of rooted cubic graphs in Lemma~\ref{lem:structure}
\label{fig:all_types}}
\end{figure}
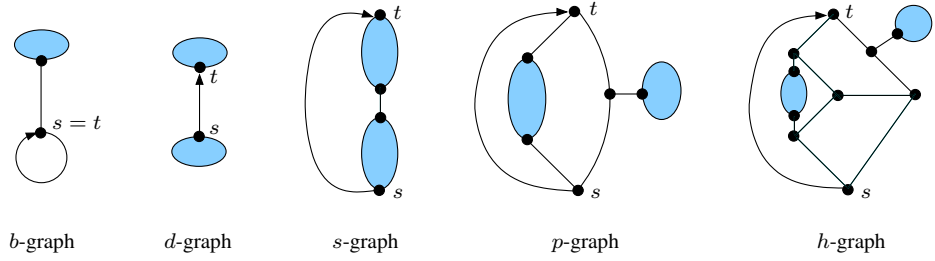
The generating function for the family $\mathcal B$ of $b$-graphs
is denoted by $B(x)$. The other generating functions
in (\ref{eq:genfunctions}) are analogously defined according
to their corresponding types.
Furthermore, the above system of equations follows
from decomposition of graphs.

Note that the number of labeled connected cubic planar weighted
multigraphs with one distinguished oriented edge is counted by
$3x\, \frac{dG^{(1)}(x)}{dx}$.

The difference between the system of equations above
and that in Bodirsky {\it et al.}~\cite{BKLM} arises as follows.
In \cite{BKLM} the term $(B(x)-x^2/4+x^2/8))$ does not appear
in $3x\, \frac{dG^{(1)}(x)}{dx}$.
The reason is that each graph in $\mathcal B$ enumerated by $B(x)$
works merely as a building block in $\hat G$ when considering simple graphs,
while it may also appear as a connected component
(whose root edge is a loop) when considering multigraphs.
Note however that when considering weights, a graph on 2 vertices in
$\mathcal B$ (that is, a rooted "dumbbell" consisting of two vertices,
an edge, and two loops one of which is marked as a root) gets
a weight 1/2 when it is used as a building block
(one loop disappears in the building operation),
but 1/4 when used as an isolated component (due to two loops).
In a similar way, a graph on 2 vertices in the family $\mathcal P$
enumerated by $P(x)$ (which is indeed a triple edge, one of which
is oriented) gets a weight 1/2 when it is used as a building block,
but 1/6 when used as an isolated component.

Analogous to Section 6 in \cite{BKLM}, one can use the singularity analysis
to obtain the following asymptotic estimates.
\begin{theorem}\label{thm:cubic-asymptotics}
For $n$ even,
\begin{align}
g^{(0)}_n  =\;  &  (1 + O(n^{-1}))\ g\ n^{-7/2}\
\rho^{-n}\ n!\,,\label{eq:generalcubic}\\
g^{(1)}_n =\;  & (1 + O(n^{-1}))\ g_c\ n^{-7/2}\
\rho^{-n}\ n!\,,\label{eq:concubic}
\end{align}
where all constants are analytically given, $\rho$ is the dominant
singularity of $G^{(1)}(x)$, and $g_c/g =e^{-G^{(1)}(\rho)}$.
Furthermore, $g_n^{(0)}=g_n^{(1)}=0$ for odd $n$.
\end{theorem}

We note that the first digits of $\gamma:=\rho^{-1}$ are  $3.38$, while
the growth constant for the labeled \emph{simple}
cubic graphs is close to $3.13$~\cite{BKLM}.
The difference in growth is due to the fact that,
unlike in the non-planar case, in average the cubic planar graphs
contain large number of multiple edges and loops.

\smallskip

Let $G_n\in \mathcal G^{(0)}_n$ denote a random graph chosen uniformly
at random from the family $\mathcal G^{(0)}_n$of all labeled cubic planar
weighted multigraphs on $n$ vertices, where each multigraph
with $f_1$ loops, $f_2$ double edges,
and $f_3$ triple edges gets weight $2^{-f_1-f_2}6^{-f_3}$.
Using the asymptotic estimation of $g^{(0)}_n$, $g^{(1)}_n$, we obtain the
following results on the size $L_1(G_n)$ of the largest component of $G_n$.
\begin{lemma}\label{lem:cubic-giant}
\begin{enumerate}
\item[(a)] Uniformly over $n=1,2,\dots$, and $0 \leq j < n/2$,
\begin{equation} \label{eqn.frag1}
   \pr( L_1(G_n) = n-j)  =
      (1+O(1/n))\ {g_c}\ (1-j/n)^{-7/2}\ j^{-7/2}.
\end{equation}
\item[(b)] There are constants $C$ and $n_0$ such that for all
          integers $n\ge n_0$ and $j=1,2,\dots$
\begin{equation} \label{eqn.frag2}
  \pr(L_1(G_n)  \le n-j) \leq C j^{-5/2}.
\end{equation}
\end{enumerate}
\end{lemma}

\begin{proof}
(a) If $L_1(G_n) = n-j$, where $j < n/2$,  then the graph obtained
from $G_n$ by deleting the largest component
(which is a connected cubic planar weighted multigraph on $n-j$ vertices)
is an arbitrary cubic planar weighted multigraphs on $j$ vertices.
By Theorem~\ref{thm:cubic-asymptotics} we have
\begin{align*}
\pr(L_1(G_n) = n-j)
=\;&  {n \choose j}\ \frac{g^{(1)}_{n-j} \times g^{(0)}_{j}}{g^{(0)}_{n}} \\
=\;& (1+O(1/n))\ {n \choose j}\ \frac{g_c\ (n-j)^{-7/2}\ \rho^{-(n-j)}\ (n-j)!
\times  g\ j^{-7/2}\ \rho^{-j}\ j!  }{g\ n^{-7/2}\ \rho^{-n}\ n! }\\
=\;&  (1+O(1/n))\ {g_c} (1-j/n)^{-7/2}\ j^{-7/2}.
\end{align*}

(b) We let  $R(G_n)$ denote $n-L_1(G_n)$, the number of vertices outside
of the largest component.
The part (a) implies that there is a constant $C'$ such that
for each $n$ and each $1 \leq j < n/2$,
\begin{equation*} \label{eqn.frag3}
  \pr(R(G_n) = j) \leq C'  j^{-7/2},
\end{equation*}
and hence
\begin{equation*}
\pr(j \leq R(G_n) < n/2) \leq   C' \sum_{j \leq i < n/2} j^{-7/2}
\leq 2C' j^{-5/2}.
\end{equation*}
Now it suffices to show that $\pr(L_1(G_n) \leq n/2) = O(n^{-5/2})$.
Let $G \in \mathcal G^{(0)}_n $ with $L_1(G) \leq n/2$.
Then, the vertex set of $G$ can be partitioned into
two sets $V_1$, $V_2$, such that $n/3 \leq |V_1|, |V_2| \leq 2n/3$,
$|V_1|+ |V_2|=n$, and there is no vertex between $V_1$ and $V_2$.
Hence,  since for all  sufficiently large $n$ we have
$ 2^{-1}  g n^{-7/2}\ \rho^{-n}\ n! \leq |g^{(0)}_n |
\leq  2 g n^{-7/2}\ \rho^{-n}\ n! $,  for large  $n$ we get
\begin{align*}
  |\{G \in \mathcal G^{(0)}_n :  L_1(G) \leq n/2 \}|
  & \leq\;
  \sum_{n/3 \leq i \leq n/2}  {n \choose i}  g^{(0)}_i \  g^{(0)} _{n-i} \\
  & \leq\;  4 g^2 \rho^{-n}\ n!
  \sum_{n/3 \leq i \leq n/2} i^{-7/2}  \ (n-i)^{-7/2}\\
& = \;  O(n^{-5/2} g^{(0)}_n)\,.
\end{align*}
This completes the proof of part (b).
\end{proof}

\subsection{Shrinking}
Let $\Q(n;d)$ denote the family of labeled planar multigraphs $G$
on vertex set $[n]=\{1,2,\dots,n\}$  with $(3n+d)/2$ edges which
have minimum degree at least three.
Therefore the deficiency $\dd(G)$ of $G\in \Q(n;d)$ equals $d$.
Moreover, let
$$\q(n;d)=\sum_{G\in \Q(n;d)}2^{-f_1(G)}\prod_{i\ge 2}(i!)^{-f_i(G)},$$
where $f_1(G)$ counts loops in $G$, and $f_i(G)$
stands for the number of  edges with parity $i$ for each $i\ge 2$.
Since each $G\in \Q(n;0)$ is cubic, the asymptotic behavior
of $\q(n;0)$ is determined by (\ref{eq:generalcubic}).

The following lemma gives bounds for $\q(n;d)$ for $d\ge 1$.

\begin{lemma}\label{lem:cubic-deficiency}
Let $1\le d\le n$ and $n>0$ be integers such that $3n+d$ is even.
Then we have
\begin{equation}\label{eq:deficiency}
\frac{\q(n+d;0)}{d! 6^{2d}}
\le\q(n;d)
\le \frac{\q(n+d;0)9^{d}}{d!}
\end{equation}
where $\q(n;0)$ is given by (\ref{eq:generalcubic}).
\end{lemma}

\begin{proof}
The idea of the proof is to generate multigraphs in $\Q(n;d)$
from the graphs from $\Q(n+d;0)$ by contracting edges
incident to vertices $\{n+1,n+2,\dots, n+d\}$.  More precisely,
for each of the vertices
$\{n+1,n+2,\dots, n+d\}$ we choose one of
the incident edges $e_i=\{i,w_i\}$,
$i=n+1,\dots, n+d$, and contract it, i.e.\ we replace
the vertices $i,w_i$ by one vertex $x$
which is adjacent to all neighbors of $i$ and $w_i$.
Finally, we  relabeled $x$ by  label $\min\{i,w_i\}$.
Sometimes this procedure fails to give
a multigraph in $\Q(n+d;0)$ (e.g.\ when an edge is nominated
by both its ends, or some of the edges $e_i$
form a cycle), nonetheless each multigraph from $\Q(n;d)$ can clearly
be obtained from some graph from $\Q(n;d)$ in the above process.

Now let us show the upper bound for $\q(n;d)$.
Choose  $G\in \Q(n+d;0)$ and  select edges
$e_i$, $i=n+1,\dots, n+d$,   in one of at most $3^d$ ways. Suppose that
by contracting all edges $e_i$, $i=n+1,\dots, n+d$,
we get a multigraph $H\in \Q(n;d)$. Note that the
weight of the multigraph $H$ could only increase by
at most $(3/2)^d$ in the case when
all vertices  $i=n+1,\dots, n+d$ belong  to different
components of size two  (then we replace a triple edge
which contributes to the weight $1/6$ by  two loops
of total weight $(1/2)^2$). Finally, we claim that
there are at least $d! 2^{-d}$ graphs $G'\in \Q(n+d;0)$ which differ
from $G$ only by labelings of vertices
$n+1,\dots, n+d$, i.e.\ each $H\in \Q(n;d)$ is counted in this procedure
at least $d! 2^{-d}$ times. Indeed, let us remove all
labels $n+1,\dots, n+d$ from vertices of $G$ getting
a graph $\bar G$ in which $d$ `dummy' vertices are not
labeled. We try to relabel those vertices with $n+1,\dots, n+d$.
Take  any vertex $w$ of $H$ of degree larger than three.
Then, $w$ is adjacent in $\bar G$ to $i$ dummy vertices,
where $1\le i\le 3$.
Thus, we can label neighbors of $w$  with $i$ labels
from $n+1,\dots, n+d$ in at least
$$\binom {d}{i}\ge \frac{d\dots (d-i+1)}{2^i}$$
ways. Now take another vertex which has been already labeled
and choose labels for its dummy neighbors, and so on,
until all dummy vertices gets their labels.
Clearly, the number of way of doing that is bounded
from below by $d!2^{-d}$.
Hence, using  (\ref{eq:generalcubic}), we get
\begin{equation*}
\q(n;d)
\le \frac{\q(n+d;0)3^d(3/2)^d}{d! 2^{-d}}
= \frac{\q(n+d;0)9^{d}}{d!}.
\end{equation*}

In order to get a lower bound for $\q(n;d)$ we count only
multigraphs $H\in \Q(n;d)$ with maximum degree four.
Note that the number of vertices of degree four in $H$ is $d$.
Each vertex  $v$ of degree four in $H$ we split into two vertices:
one of them we label with $v$ the other we leave as a `dummy'
vertex which has not get  labeled so far.
We add an edge between the two vertices, which we mark as used.
We can make such a split into at most six possible ways.
Now, we can choose
labels for dummy vertices into one of possible $d!$ ways
(note that each dummy vertex is uniquely identified
by the other end of the used edge). Consequently,
from each $H$ we get at most $d! 6^d$ different
graphs from $\Q(n;d)$ with $d$ disjoint edges marked.
Note however that splitting a vertex we may increase
the weight of the graph by at most six (if we
split the quadruple edge into two double edges),
thus the total weight of different multigraphs
obtained from $H$ is bounded from above by $d! 6^{2d}$.
Consequently,
\begin{equation*}
\q(n;d)\ge \frac{\q(n+d;0)}{d! 6^{2d}}\,.
\end{equation*}
\end{proof}

\section{Planar graphs with positive excess}\label{sec:complexplanar}
Recall that a graph is called complex if all its components have
positive excess.
In this section, we  derive the asymptotic number of all labeled
\emph{complex} planar graphs  with given size and excess
(Theorem~\ref{thm:complex}). In order to do that we first estimate
the number of such graphs with given deficiency
(Lemma~\ref{lem:complex-supercubickernel}).

\subsection{Fixed deficiency}
In this section we  estimate the number $C_{d}(k,k+\ell)$ of all labeled
complex planar graphs $G$ on $k$ vertices with $\ex(G)=\ell>0$
and $\dd(G)=d\ge 0$.

To this end, observe first that the core $\core(G)$
of a complex graph~$G$ can be  obtained
from~$G$ by pruning each vertex of degree one in $G$ recursively,
i.e. to get $\core(G)$ we have to delete the tree-like part of~$G$
rooted at vertices of $\core(G)$. Then, in order  to find the kernel
$\ker(G)$ of $G$, one needs to replace each path in  $\core(G)$
whose internal vertices are all of degree two by a single edge.
Note that both $\core(G)$ and $\ker(G)$ have the same excess as $G$,
and  $\ker(G)$  is a planar multigraph of minimum degree at least three.
Therefore, if $\ex(G)=\ell$ and $\dd(G)=d$,  the number of vertices
in $\ker(G)$, denoted by $v(\ker)$,  equals $2\ell-d$ and the number
of edges, denoted by $e(\ker)$,  equals $3\ell-d$.

In order to find $\pl(n,M)$ we  reverse the above procedure
and first count all possible kernels of graphs, then
study the number of cores which lead to these kernels, and finally
add to them the rooted forest to obtain all possible graphs $G$.
More precisely, one can construct all graphs $G$
on $k$ vertices with $k+\ell$ edges, whose kernel is a planar multigraph
graph with minimum degree at least three and deficiency $d$,
in the following way.
\begin{itemize}
\item [(i)] Choose the vertex set $\{c_1,c_2,\cdots,c_i\}$ of
the core  (for some $i\le k$), and then select the vertex set
$\{k_1,k_2,\cdots,k_{v(\ker)}\}$ of the kernel
from the vertex set of the core.
It can be done in ${k \choose i}{i\choose v(\ker) }$ ways.
\item [(ii)] Select a kernel  of order $v(\ker)$  among all
the possible candidates for kernels (i.e.\ cubic planar weighted
multigraphs on vertex set $[v(\ker)]$ with $e(\ker)$ edges),
and then map  $[v(\ker)]$ to $\{k_1,k_2,\cdots,k_{v(\ker)}\}$
in their relative order, i.e.\ we map $j$ to $k_j$.
There are $\q(v(\ker);d)$ ways of doing that.

\item [(iii)]  Order the edges of the kernel lexicographically,
each edge with a direction from the one end point with the smaller label
to the other end point with the larger label. For multiple edges with
the same ends take any order, and choose one of two possible
directions  for each loop (the weights we assign when we
counted candidates for the kernel was chosen precisely
to assure that in this way we avoid double counting).
Now make a directed path of length $e(\ker)$ consisting of
the kernel edges, according to this order, and insert the core vertices
that are not in the kernel on the edges of the kernel in such a way that
each loop gets at least two core vertices, and at least $j-1$ edges
from $j$ edges incident to the common end points
get at least one core vertex.
Let $m=m(\ell)$ be such that $m\, \ell=2f_1+\sum_{j\ge 2}(j-1)f_j$,
where $f_1$ denotes the number of loops and $f_j$ multiple edges
with parity $j$ for $j\ge 2$ in $\ker(G)$ (hence $0\le m\le 6$).
This can be done in
$(i-v(\ker))!\binom{i-v(\ker)-m\ \ell+e(\ker)-1}{e(\ker)-1}$
ways.
\item [(iv)] Plant a rooted forest on the core vertices.
According  to Cayley's formula, one can do it in $ik^{k-i-1}$ ways.
\end{itemize}
As a consequence, we have
\begin{align}
&C_{d}(k,k+\ell)\nonumber\\
&=\sum_{i}{k \choose i}{i\choose v(\ker)}\ \q(v(\ker);d)\
 (i-v(\ker))!\binom{i-v(\ker)-m\ \ell+e(\ker)-1}{e(\ker)-1}\
 ik^{k-i-1}\nonumber\\
&=\sum_{i}\frac{(k)_i}{(2\ell-d)!}\ \q(2\ell-d;d)\ \binom{i+\ell-m\
\ell-1}{3\ell-d-1}\ ik^{k-i-1}.\label{eq:Cd}
\end{align}

Applying Lemma~\ref{lem:cubic-deficiency} we obtain the following estimate.

\begin{lemma}\label{lem:complex-supercubickernel}
Let $d\ge 0$ and $k, \ell>0$ be integers.
Let $\gamma$, $g$ be the constants so that the assertion of
Theorem~\ref{thm:cubic-asymptotics} holds. Then
\begin{equation}\label{eq:Cd-final}
\begin{aligned}
C_{d}(k,k+\ell)
&=2^{-4}g\ k^{k+\frac{3\ell-d-1}{2}}\ {\gamma}^{2\ell}\ \ell^{-7/2}\
e^{\frac{3\ell-d}{2}} (3\ell-d)^{-\frac{3\ell-d-1}{2}}
\alpha^{d}\binom{2\ell}{d}\nonumber\\
&\times
\exp\left(\left(\frac{2}{3}d-(m+1)\ell\right)\sqrt{\frac{3\ell-d}{k}}
+O\left(\frac{\ell^2}{k}+\frac{1}{\ell}\right)\right),
\end{aligned}
\end{equation}
for some $\alpha=\alpha(k,\ell)$, $6^{-2}\le \alpha\le 9$,
and $m=m(k,\ell)$, $0\le m\le 6$.

Moreover, the typical size of the core in a randomly chosen complex
planar graph $G$ on $k$ vertices with $\ex(G)=\ell$ and $\dd(G)=d$ is \aas
$$\big(1+O(\sqrt{\ell/k})+O(1/\sqrt{\ell})\big)\sqrt{3k\ell}\,.$$
\end{lemma}

\begin{proof}
Note that from Lemma \ref{lem:cubic-deficiency} we have
\begin{equation*}
\frac{\q(2\ell;0)}{d! 6^{2d}}
\le \q(2\ell-d;d)
\le
\frac{\q(2\ell;0)9^d}{d!},
\end{equation*}
where
\begin{equation*}
\q(2\ell;0) \, \stacksign{(\ref{eq:generalcubic})}{=}\,
(1+o(\ell^{-1}))g \ {(2\ell)}^{-7/2}\ {\gamma}^{2\ell}\  (2\ell)!.
\end{equation*}
Therefore from  (\ref{eq:Cd}) we get
\begin{align}
&\hspace{-5ex}\sum_{2\ell-d\le i \le k }g\ k^{k-1}\
{\gamma}^{2\ell}\ (2\ell)^{-7/2}\ {6^{-2d}}\ \binom{2\ell}{d}\
\sum_{i}\ (k)_i\ i\ k^{-i}\ \binom{i+(1-m)\ell-1}{3\ell-d-1}
\le C_{d}(k,k+\ell)\nonumber\\
&\le \sum_{2\ell-d\le i \le k }g\ k^{k-1}\ {\gamma}^{2\ell}\
(2\ell)^{-7/2}\ {9^{d}}\ \binom{2\ell}{d}\
\sum_{i}\ (k)_i\ i\ k^{-i}\ \binom{i+(1-m)\ell-1}{3\ell-d-1}.\label{eq:Cd-2}
\end{align}
Thus, it is enough to estimate the quantity
\begin{equation}\label{eq:cdeficiency-upper}
\tilde C_{d}(k,k+\ell)
=g\ k^{k-1}\ {\gamma}^{2\ell}\ (2\ell)^{-7/2}\
{\alpha^{d}}\ \binom{2\ell}{d}\
\sum_{i}\ (k)_i\ i\ k^{-i}\ \binom{i+(1-m)\ell-1}{3\ell-d-1},
\end{equation}
where  $\alpha$ and $m$ satisfy $6^{-2}\le \alpha\le 9$ and $0\le m\le 6$.

Below  we use several times Stirling's formula
\begin{equation}\label{eq:stirling}
n! =\; (1+O(1/n))\sqrt{2\pi}\ n^{n+1/2}e^{-n}\quad
\textrm{for}\quad n\in \mathbb N
\end{equation}
and the following consequence of Maclaurin expansion of $e^x$
\begin{equation}\label{eq:1+x}
1+x = \exp(x-x^2/2+x^3/3+O(x^4))\,.
\end{equation}
To derive an asymptotic formula for (\ref{eq:cdeficiency-upper}),
note that
\begin{align*}
(k)_i
= k^i\prod_{j=0}^{i-1}\left(1-\frac{j}{k}\right)
&\stacksign{(\ref{eq:1+x})}{=}\,
k^i\exp\left(\sum_{j=0}^{i-1}\left({-\frac{j}{k}- \frac12
\left(\frac{j}{k}\right)^2- \frac13
\left(\frac{j}{k}\right)^3+O\left(\frac{j}{k}\right)^4}\right)\right)\nonumber\\
&=k^i\exp\left({-\frac{i^2}{2k}-\frac{i^3}{6k^2}+
O\left(\frac{i}{k}+\frac{i^4}{k^3}\right)}\right)\,
\end{align*}
and
\begin{align*}
&\hspace{-5ex}\binom{i+(1-m)\ell-1}{3\ell-d-1}
= \frac{(i+(1-m)\ell-1)_{3\ell-d-1}}{(3\ell-d-1)!}\\
&\; \stacksign{(\ref{eq:stirling})}{=}
\frac{(3\ell-d)\ i^{3\ell-d-1}\prod_{j=1}^{3\ell-d-1}
\left(1+\frac{(1-m)\ell -j}{i}\right)}{(1+O\left({1}/{\ell}\right))
\sqrt{2\pi (3\ell-d)}\ ((3\ell-d)/e)^{3\ell-d}}\nonumber\\
&\stacksign{(\ref{eq:1+x})}{=}
\frac{e^{3\ell-d}i^{3\ell-d-1}}{\sqrt{2\pi}(3\ell-d)^{3\ell-d-1/2}}
\exp\left(-\frac{(3\ell-d)^2}{2i}+\frac{(1-m)\ell(3\ell-d)}{i}
+\frac{1}{\ell}+\frac{\ell}{i}\right).%\label{eq:binom}
\end{align*}
Next, we rewrite the sum over $i$ in (\ref{eq:cdeficiency-upper}) as
\begin{equation}\label{eq:iterms}
\sum_{i}\ (k)_i\ i\ k^{-i}\ \binom{i+(1-m)\ell-1}{3\ell-d-1}
=\frac{e^{3\ell-d}}{\sqrt{2\pi}(3\ell-d)^{3\ell-d-1/2}}
\ \sum_{i}\exp( a(i))\,,
\end{equation}
where the function $a(i)=a_{k,\ell,d}(i)$ is defined as
\begin{align*}
a(i)= (3\ell-d)\log i
&{-\frac{i^2}{2k}-\frac{i^3}{6k^2}}
-\frac{(3\ell-d)^2}{2i}+\frac{(1-m)\ell(3\ell-d)}{i}+\frac{1}{\ell}\\
&+\frac{\ell}{i}+O\left(\frac{i}{k}\right)+O\left(\frac{i^4}{k^3}\right).
\end{align*}
We observe that the main contribution to  (\ref{eq:iterms})
comes from the terms $i=i_0+O(\sqrt{k})$,  where
\begin{equation}\label{eq:add1}
i_0=\left(1+\frac{m-1}{6}\sqrt{\frac{3\ell-d}{k}}\right)\sqrt{k(3\ell-d)}.
\end{equation}
For such $i$'s, we have
\begin{align*}
\exp(a(i_0))
=(k(3\ell-d))^{\frac{3\ell-d}{2}}&\ \exp\left( -\frac{3\ell-d}{2}
+\left(\frac{2}{3}d-(m+1)\ell\right)\sqrt{\frac{3\ell-d}{k}}\right)\\
&\quad\quad\times \exp\left(O\left(\frac{\ell^2}{k}\right)
+O\left(\frac{1}{\ell}\right)\right)
\end{align*}
and
\begin{equation*}
\sum_{i=i_0+O(\sqrt{k})}\exp(a(i)-a(i_0))
=\sum_{\Delta i=O(\sqrt{k})}\exp\left(- \frac{1}{k}(\Delta i)^2
+ O\left(\frac{\ell}{k}\right)\right)%\\
={\sqrt{\pi k}}\ \exp\left(O\left(\frac{\ell}{k}\right)\right).
\end{equation*}
Thus, we get
\begin{equation}\label{eq:iterms-final}
\begin{aligned}
\hspace{-5ex}\sum_{i}\ (k)_i\ i\ k^{-i}\binom{i+(1-m)\ell-1}{3\ell-d-1}
&= 2^{-1/2}  k^{\frac{3\ell-d+1}{2}}
(3\ell-d)^{-\frac{3\ell-d+1}{2}} e^{\frac{3\ell-d}{2}}\\
&\hspace{-15ex}\times \exp\left(\left(\frac{2}{3}d-(m+1)\ell\right)
\sqrt{\frac{3\ell-d}{k}} +O\left(\frac{\ell^2}{k}
+\frac{1}{\ell}\right)\right).
\end{aligned}
\end{equation}
Finally, (\ref{eq:cdeficiency-upper}) and (\ref{eq:iterms-final}) yield
\begin{align*}
\tilde C_{d}(k,k+\ell)
=\,  2^{-4}g\ {\gamma}^{2\ell}
&\ \ell^{-7/2}\ k^{k+\frac{3\ell-d-1}{2}}\ e^{\frac{3\ell-d}{2}}
(3\ell-d)^{-\frac{3\ell-d-1}{2}} \alpha^{d}\binom{2\ell}{d}\\
&\times \exp\left(\left(\frac{2}{3}d-(m+1)\ell\right)
\sqrt{\frac{3\ell-d}{k}}+O\left(\frac{\ell^2}{k}+\frac{1}{\ell}\right)\right).
\end{align*}
The last part of the assertion follows from the fact that
the main contribution to the sum (\ref{eq:cdeficiency-upper})
comes from the
terms $i=(1+o(1))i_0$, where $i_0$ is given by (\ref{eq:add1}).
\end{proof}

\subsection{Asymptotic numbers and typical deficiency}
In this section we  estimate the number $C(k,k+\ell)$
of labeled complex planar graphs with $k$ vertices
and $k+\ell$ edges.

\begin{theorem}\label{thm:complex}
Let $\gamma$, $g$, $g_c$ be the constants for which the assertion
of Theorem~\ref{thm:cubic-asymptotics} holds and
let $k$, $\ell>0$ be integers.
\begin{itemize}
\item[(i)]
There exists
a function $\beta=\beta(k,\ell)$ with $-14\le \beta\le 2^7$, for which
\begin{align*}
C(k,k+\ell)
=2^{-4} 3^{1/2}\ g &\  k^{k+3\ell/2-1/2}
\left(\frac{\gamma^{2}e^{3/2}}{3^{3/2}}\right)^{\ell}
\ \ell^{-3\ell/2-3}\\
& \times
\exp\left(\beta\sqrt{\frac{\ell ^3}{k}}
+O\left(\frac{\ell^2}{k}\right)+O\left(\frac{1}{\ell}\right)\right).
\end{align*}
\item[(ii)] The number $C^{conn}(k,k+\ell)$
of labeled \emph{connected} complex planar graphs
with $k$ vertices and $k+\ell$ edges is given by  a similar
formula, with $g$ replaced by $g_c$.
\item[(iii)] A graph chosen uniformly at random among all complex planar
 graphs with $k$ vertices and $k+\ell$ edges has \aas deficiency
 $\Theta(\sqrt{\ell^3/k})$ and the core of size
$(1+O(\sqrt{\ell/k})+O(1/\sqrt{\ell}))\sqrt{3k\ell}$.
In particular, if $\ell=o(k^{1/3})$,
then \aas such a random graph is clean.

\item[(iv)] If $\ell=O(k^{1/3})$, then  a graph chosen uniformly at
 random among all complex planar graphs with $k$ vertices and $k+\ell$
 edges has \aas $O(1)$ components, among which there is a
 giant component of size  $k-O(k/\ell)$. Furthermore, \aas
each  small component has $\Theta(k/\ell)$ vertices and
the probability that such a graph contains exactly $h$ such components
is bounded away from both 0 and 1 for every $h=0,1,\dots$.
\end{itemize}
\end{theorem}

\begin{proof}
Using the asymptotic estimate of $C_{d}(k,k+\ell)$ from
Lemma~\ref{lem:complex-supercubickernel}, we get the following.
\begin{align}
C(k,k+\ell)
&=\sum_{d}C_{d}(k,k+\ell)\nonumber\\
&\stacksign{(\ref{eq:Cd-final})}{=}\,
2^{-4}g\ {\gamma}^{2\ell}\ k^{k+\frac{3\ell-1}{2}}\
e^{\frac{3\ell}{2}} (3\ell)^{-\frac{3\ell-1}{2}}\ \ell^{-7/2}\
\exp\left(O\left(\frac{\ell^2}{k}
+\frac{1}{\ell}\right)\right)\label{eq:maxdeficiency}
\\
&\times
\sum_{d}k^{-\frac{d}{2}}\ (3\ell)^{\frac{d}{2}} \alpha^{d}\binom{2\ell}{d}
\exp\left(\left(\frac{2}{3}d
-(m+1)\ell\right)\sqrt{\frac{3\ell-d}{k}}\right),\nonumber
\end{align}
where $\alpha=\alpha(k,\ell)$ and $m=m(k,\ell)$ satisfy
$6^{-2}\le \alpha\le 9$ and $0\le m\le 6$.

Define a function $\eta(d)=\eta_{k,\ell}(d) $ as
\begin{equation*}
\eta(d) = \sum_{d}k^{-\frac{d}{2}}\ (3\ell)^{\frac{d}{2}}
 \alpha^{d}\binom{2\ell}{d}
\exp\left(\left(\frac{2}{3}d-(m+1)\ell\right)\sqrt{\frac{3\ell-d}{k}}\right).
\end{equation*}
We observe that the main contribution to $\eta(d) $ comes from
$d=\Theta(\sqrt{\ell^3/k})$ and therefore
\begin{align*}
\eta(d)
&= \sum_{d} \binom{2\ell}{d} (\alpha \sqrt{3\ell/k})^{d}
\exp\left(-(m+1)\ell\sqrt{\frac{3\ell}{k}}+O\left(\frac{\ell^2}{k}
\right)\right)\\
&=  (1+\alpha \sqrt{3\ell/k})^{2\ell}
\exp\left(-(m+1)\sqrt{\frac{3\ell^3}{k}}+O\left(\frac{\ell^2}{k}\right)\right)\\
&=  \exp\left((2\alpha-(m+1))
\sqrt{\frac{3\ell^3}{k}}+O\left(\frac{\ell^2}{k}\right)\right).\nonumber
\end{align*}
Finally, taking $\beta=\beta(k,\ell)=(2\alpha-(m+1))\sqrt{3}$
(and thus $-14\le \beta\le 2^7$) completes the proof of (i).

In order to show (ii) one should repeat  computations from the proof of
Lemma~\ref{lem:complex-supercubickernel} and the one given above,
for graphs with \emph{connected} kernels.
Therefore $\q(2\ell;0)$ should be replaced by the number of \emph{connected}
cubic planar weighted multigraphs vertices, which,  by (\ref{eq:concubic}),
is equal to
$(1+O(\ell^{-1}))g_c \ {(2\ell)}^{-7/2}\ {\gamma}^{2\ell}\  (2\ell)!$.

To see (iii) observe the main contribution to (\ref{eq:maxdeficiency})
comes from $d=\Theta(\sqrt{\ell^3/k})$.

Finally,  Lemma~\ref{lem:cubic-giant} states that  a randomly chosen
cubic planar graph \aas contains a giant component of size $n-O(1)$,
and using exact counts it is easy to show that the number of small
components has a non-degenerate distribution. Since for $\ell=O(k^{1/3})$
we can count graphs up to a constant factor, a similar statement is true
also for supercubic weighted multigraphs.
Now (iv) follows from the fact that
the trees rooted in one edge of the kernel have in average  $\Theta(k/\ell)$
vertices altogether.
\end{proof}

Unfortunately, since we can only estimate the number of supercubic graphs
up to a factor of  $\exp(\sqrt{\ell^3/k})$, we cannot prove the assertion
of Theorem~\ref{thm:complex}(iv) in the case when $\ell\gg k^{1/3}$.
Nonetheless, we think that it is true also in much wider range
and that the following conjecture holds.

\medskip

{\bf Giant Conjecture.} {\sl The assertion of  Theorem~\ref{thm:complex}(iv)
holds for every $\ell\le k$.}

\section{Evolution of planar graphs}
In this section we  derive the asymptotic number $\pl(n,M)$
of labeled planar graphs with $n$ vertices and $M$ edges and investigate
how the size of the largest component in $P(n,M)$, its excess and
the size of its core change with $M$.

Throughout the section, we let $\gamma$, $g$ be the constants for which
the assertion of  Theorem~\ref{thm:cubic-asymptotics} holds.
By $L_j(n,M)$ we denote the number of vertices in the $j$-th largest
component of $P(n,M)$.
Let $\ex_{\rm c}(n,M)$ (resp. $\Pcr_{\rm c}(n,M)$) stand for the excess
(resp. the number of vertices in the core) of the subgraph of $P(n,M)$
which consists of its complex
components, and let  $L_{\rm c}(n,M)$ denote its size. Finally,
let $\ex(n,M)$ and $\Pcr(n,M)$ denote the excess and
the size of the core of the largest component of $P(n,M)$, respectively.

Before studying $P(n,M)$ we recall in the next section some
properties of the uniform random graph $G(n,M)$
which are relevant for our argument.

\subsection{Properties of the uniform random graph}
Let $\bar L_j(n,M)$ denote the number of vertices
in the $j$-th largest component of $G(n,M)$,
let $\Gex(n,M)$ stand for the excess of the largest component
of $G(n,M)$, and let $\Gcr(n,M)$ be the number of vertices in the core
of the largest component of $G(n,M)$.

The following results on the largest components were proved
by {\L}uczak~\cite{L,L96} and {\L}uczak, Pittel, and Wierman~\cite{LPW}
(see also Janson, {\L}uczak, and Ruci\'nski~\cite{JLR}).

\begin{theorem}[Subcritical phase]\label{thm:subGnM}
Let $M=n/2+s$, where $s=s(n)$.
If $s^3/n^2\to -\infty$, then for $j$ fixed, \aas
$\bar L_j(n,M)=(1/2+o(1))\frac{n^2}{s^2} \log \frac{|s|^3}{n^2}$.
Furthermore, \aas the $j$-th largest component of $G(n,M)$ is a tree.
\end{theorem}

\begin{theorem}[Critical phase]\label{thm:criticalGnM}
Let $M=n/2+s$, where $s=s(n)$.
If $s^3/n^2\to c$, then for $j$ fixed, \aas $\bar L_j(n,M)=\Theta(n^{2/3})$.
The total excess of the complex components of $G(n,M)$ is \aas $O(1)$,
and the probability that the $j$-th largest component
of $G(n,M)$ has excess $h$
is bounded away from zero for every fixed $j=1,2,\dots$,
and $h=-1,0,1,\dots$.

Furthermore, if $G(n,M)$ contains some complex components, then
\aas they have $\Theta(n^{2/3})$ vertices in total
and $\Gcr(n,M)=\Theta(n^{1/3})$.
\end{theorem}

\begin{theorem}[Supercritical phase]\label{thm:supercriticalGnM}
Let $M=n/2+s$, where $s=s(n)$.
If $s^3/n^2\to \infty$, then \aas $\bar L_1(n,M)=(4+o(1))s$,
while for $j\ge 2$ fixed,
$\bar L_j(n,M)=(1/2+o(1))\frac{n^2}{s^2} \log \frac{s^3}{n^2}$.
Moreover, \aas the $j$-th largest component of $G(n,M)$ is a tree,
provided $j\ge 2$.
\end{theorem}

The structure of the giant component of $G(n,M)$
was studied by {\L}uczak~\cite{cycle}.
\begin{theorem}\label{thm:GnMexcesscore}
If $M=n/2+s$, where $s^3/n^2\to \infty$ but $s=o(n)$,
then \aas $\Gex(n,M)= (16/3+o(1)) \frac{s^3}{n^2}$
and $\Gcr(n,M)=(8+o(1)) \frac{s^2}{n}$.
\end{theorem}

The following threshold for the property that $G(n,M)$ is planar
 was proved by {\L}uczak, Pittel, and Wierman~\cite{LPW}.
\begin{theorem}[Planarity]\label{thm:planariy}
Let $M=n/2+cn^{2/3}$ for a constant $c$.
Then the probability that $G(n,M)$ is planar tends to a limit $\varphi(c)$
as $n\to \infty$, where  $0<\varphi(c)<1$,
$\lim_{c\to -\infty} \varphi(c)=1$, and $\lim_{c\to \infty} \varphi(c)=0$.
\end{theorem}

A different proof of the above result can be found in Janson
{\it et al.}~\cite{JKLP}, who also showed that $0.987<\varphi(0)<0.9998$.

We shall use Theorems~\ref{thm:subGnM},~\ref{thm:criticalGnM}
and~\ref{thm:planariy} in the proofs of Theorems~\ref{thm:subcritical}
and~\ref{thm:critical}.
We do not use Theorems~\ref{thm:supercriticalGnM}
and~\ref{thm:GnMexcesscore} in our proofs below, but
we decide to invoke them here to show
how our results differ from those for $G(n,M)$.

For a constant $c\in (-\infty,\infty)$, let us define
\begin{equation}\label{def:nu}
\nu(c)=\sqrt{\frac{2}{3\pi}}e^{-4c^3/3}
\sum_{r=0}^\infty \frac{(-9c^3)^{r/3}}{r!}
\Gamma\left(\frac{2r}3+\frac12\right)\cos\frac{\pi r}{3}\;.
\end{equation}
Note that $\nu(c)$ decreases monotonically with $\nu(c)\to 1$ as
$c \to -\infty$ and
$\nu(c)\le \exp(-(4+o(1))c^3/3)$ for large $c$.
In our argument we use also the following result
of Britikov~\cite{Britikov}.
Here and below by $\rho(n,M)$ we denote the probability
that $G(n,M)$ contains no complex component.
In other words, $\rho(n,M)=U(n,M)/\binom{\binom n2}{M}$,
where $U(n,M)$ denotes the number of labeled graphs
with $n$ vertices
and $M$ edges, which contain no complex components.

\begin{theorem}\label{thm:Britikov}
Let $M=n/2+s$, where $s=s(n)$. Then the following holds.
\begin{enumerate}
\item If $s^3/n^2\to -\infty$, then
$\rho(n,M)=1+O(n^2/|s|^3)$.

\item If $s^3/n^2\to c$, where $c$ is a (not necessarily
         positive) constant,
then $\rho(n,M)=(1+o(1))\nu(c^{1/3})$.
\item If $s^3/n^2\to \infty$, then
$\rho(n,M)\le \exp(-{s^3}/{n^2})$.
\end{enumerate}
\end{theorem}

\subsection{The formula for $\pl(n,M)$}
The main ingredient of our argument is a simple observation
that each  graph can be uniquely decomposed into the complex part and
the remaining part  which consists of isolated tees and unicyclic
components. Moreover, it is the complex part which determines
whether the graph is planar.
Consequently,  the number $\pl(n,M)$ of labeled planar graphs
on $n$ vertices with $M$ edges is given by
\begin{equation}\label{eq:main}
\pl(n,M) = \sum_{k,\ell}\binom{n}{k}C(k,k+\ell)U(n-k,M-k-\ell).
\end{equation}
Thus, the estimate of $C(k,k+\ell)$ (Theorem~\ref{thm:complex})
and that of $U(n-k,M-k-\ell)$ (Theorem~\ref{thm:Britikov}) yield
the asymptotic estimate of $\pl(n,M)$. Moreover, the leading terms
of (\ref{eq:main}) give us information on the size of the complex part
of the graph and thus, by Theorem~\ref{thm:complex}, on the size
of the largest component of $P(n,M)$ and its internal structure.
On the other hand, the size of the largest non-complex
component can be deduced from Theorems~\ref{thm:subGnM}
and~\ref{thm:criticalGnM}.

\subsection{Subcritical phase}\label{sec:subcritical}
The behavior of $P(n,M)$ in the subcritical case follows directly
from Theorem~\ref{thm:subGnM}.

\begin{theorem}\label{thm:subcritical}
Let $M=n/2+s$, where $s=o(n)$.
If $s^3/n^2\to -\infty$, then
\begin{equation*}
\pl(n,M)
=(1+o(1))\frac{n^{n+2s}e^{n/2+s-1/2}}{\sqrt \pi (n+2s)^{n/2+s+1/2}}.
\end{equation*}
Furthermore,  for fixed $j$
\aas $L_j(n,M)=(1/2+o(1))\frac{n^2}{s^2} \log \frac{|s|^3}{n^2}$
and the $j$-th largest component of $P(n,M)$ is a tree.
\end{theorem}
\begin{proof}
Theorem~\ref{thm:planariy} states that \aas $G(n,M)$ is planar.
Therefore, we have
\begin{equation*}
\pl(n,M)
=(1+o(1)) \binom{\binom n2}{M}
=(1+o(1))\frac{n^{n+2s}e^{n/2+s-1/2}}{\sqrt \pi (n+2s)^{n/2+s+1/2}},
\end{equation*}
and the structure of $P(n,M)$  follows from Theorem~\ref{thm:subGnM}.
\end{proof}

\subsection{Critical phase}\label{sec:critical}
The critical period is only slightly harder to deal with
than the previous one as far as we estimate $\pl(n,M)$
only up to a constant factor.

\begin{theorem}\label{thm:critical}
Let $M=n/2+s$. If $s^3/n^2\to c$ for a constant $c\in (-\infty,\infty)$,
then
\begin{equation*}
\pl(n,M)
=\Theta(1)\frac{n^{n+2s}e^{n/2+s-1/2}}{\sqrt \pi (n+2s)^{n/2+s+1/2}}.
\end{equation*}
Furthermore, for $j$ fixed, \aas $L_j(n,M)=\Theta(n^{2/3})$,
and, if $P(n,M)$ contains complex components, then
$L_{\rm c}(n,M)=\Theta(n^{2/3})$,  $\ex(n,M)=O(1)$,
and $\Pcr(n,M)=\Theta(n^{1/3})$.
\end{theorem}

\begin{proof}
Theorem~\ref{thm:planariy} states that the probability that $G(n,M)$
is planar tends to a limit which is strictly between 0 and 1. Hence
\begin{equation*}
\pl(n,M)=\Theta(1)\binom{\binom n2}{M}
=\Theta(1)\frac{n^{n+2s}e^{n/2+s-1/2}}{(n+2s)^{n/2+s+1/2}}.
\end{equation*}
The assertion on the structure of $P(n,M)$ is a direct
consequence of Theorem~\ref{thm:criticalGnM}.
\end{proof}

\subsection{Supercritical phase}\label{sec:supercritical}
The evolution of $P(n,M)$ in the `early supercritical' period
starts to be more interesting. Note that the result below
estimates $\pl(n,M)$ up to a factor of $1+o(1)$.

\begin{theorem}\label{thm:super}
Let $M=n/2+s$, where $s=o(n)$ and $s^3/n^2\to\infty$.
Then
\begin{align*}
\pl(n,M)
=(1+o(1))&
\frac{g 3^{5/2}}{2^{7}\sqrt{\pi} \gamma^{10/3}e^{3/4}}
\frac {n^{n+11/6}}{s^{7/2}}\frac{e^{n/2-s}}{(n-2s)^{n/2-s}}
\exp\left(\frac{\gamma^{4/3}s}{n^{2/3}}\right)\\
&\quad\times  \int_{-\infty}^\infty
\exp\left(-\frac{x^3}{6}+\frac{\gamma^{4/3}x}{2}\right)
\nu\left(-\frac x2\right)dx.
\end{align*}
Furthermore, \aas $L_1(n,M)=(2+o(1))s$, while
for any fixed $j\ge 2$ we have $L_j(n,M)=\Theta(n^{2/3})$.
In addition, \aas
$\ex(n,M)= (\frac{2\gamma^{4/3}}{3}+o(1))\frac{s}{n^{2/3}}$
and $\Pcr(n,M)=(2\gamma^{2/3}+o(1)) \frac{s}{n^{1/3}}$.
\end{theorem}

\begin{proof}
From Theorem~\ref{thm:complex} and (\ref{eq:stirling}), we get
\begin{align*}
\binom{n}{k}&\; \stacksign{(\ref{eq:stirling})}{=}\;
\frac{(1+O(1/k))}{\sqrt{2\pi}}\frac{n^{n+1/2}}{(n-k)^{n-k+1/2}k^{k+1/2}},\\
C(k,k+\ell) &= \frac{g\ 3^{1/2}}{2^4}\
k^{k+3\ell/2-1/2}\left(\frac{\gamma^{2}e^{3/2}}{3^{3/2}}\right)^{\ell}\
\ell^{-3\ell/2-3}\\
& \times
\exp\left(\beta\sqrt{\frac{\ell ^3}{k}}
+O\left(\frac{\ell^2}{k}\right)+O\left(\frac{1}{\ell}\right)\right).
\end{align*}
In addition, the estimate
\begin{equation*}
\binom{\binom{n}{2}}{j} \;=\;
\frac{n^{2j}}{\sqrt{\pi} (2j)^{j+1/2}}
\exp\left(j-\frac{j}{n}-\frac{j^2}{n^2}+O\left(\frac{1}{n}\right)
+O\left(\frac{j}{n^2}\right)\right)\label{eq:nchoosek2}
\end{equation*}
and Lemma~\ref{thm:Britikov} give
\begin{align*}
U&(n-k,M-k-\ell)\\
&=\;  \rho(n-k,M-k-\ell)\binom{\binom{n-k}{2}}{M-k-\ell}\\
&=\;
(1+O(1/n))\frac{ \rho(n-k,n/2+s-k-\ell)}{\sqrt{\pi} e^{3/4}}\
  \frac{e^{n/2+s-k} (n-k)^{n+2s-2k}} {(n+2s-2k)^{n/2+s-k+1/2}}
  \left(\frac{n+2s-2k}{(n-k)^2}\right)^{\ell}.
\end{align*}
Therefore, we get
\begin{align}
\pl(n,M)
&=\sum_{k,\ell}\binom{n}{k}C(k,k+\ell)U(n-k,M-k-\ell)\nonumber\\
&=(1+O(1/n))
\frac{g\ 3^{1/2}}{2^{9/2} {\pi} e^{3/4}}
n^{n-1/2}e^{n/2+s}\nonumber\\
&\hspace{-5ex}\times \sum_k (1+O(1/k)+O(k/n)) \frac{\rho(n-k,n/2+s-k)\
(n-k)^{2s-k}}{ke^k(n+2s-2k)^{n/2+s-k}}\label{eq:firstrange1}\\
&\hspace{-5ex}\times \sum_{\ell}
\left(\frac{\gamma^{2}e^{3/2}k^{3/2}(n+2s-2k)}{3^{3/2}(n-k)^2
\ell^{3/2}}\right)^{\ell}  \ell^{-3}
\exp\left(\beta\sqrt{\frac{\ell ^3}{k}}
+O\left(\frac{\ell^2}{k}\right)+O\left(\frac{1}{\ell}\right)\right).
\nonumber
\end{align}

Now let
$\phi=\phi(n,s,k):=\frac{\gamma^{2}e^{3/2}k^{3/2}(n+2s-2k)}{3^{3/2}(n-k)^2}$.
Then the sum in (\ref{eq:firstrange1}) depending on $\ell$ becomes
\begin{equation}\label{eq:planar-lterm}
\sum_{\ell\ge 1}\left(\frac{\phi}{\ell^{3/2}}\right)^\ell \ell^{-3}
=\sum_{\ell\ge 1}\ell^{-3} \exp(b(\ell)),
\end{equation}
where the function $b(\ell)=b_{n,s,k}(\ell)$ is defined as
\begin{equation*}
b(\ell)=\ell \log \phi - \frac{3}{2} \ell \log \ell.
\end{equation*}
The main contribution to (\ref{eq:planar-lterm}) comes from
the terms $\ell=\ell_0+O(\sqrt{\ell_0})$, where
\begin{equation}\label{eq:l}
\ell_0=\ell_0(n,s,k):= e^{-1} \phi^{2/3}
=\frac{\gamma^{4/3}}{3} \frac{k(n+2s-2k)^{2/3}}{(n-k)^{4/3}}.
\end{equation}
Furthermore, we have
\begin{align*}
 \ell_0^{-3}\exp(b(\ell_0))
&=\,  \ell_0^{-3}\exp(3\ell_0 (\log \phi^{2/3} - \log \ell_0)/2) \\
&=\, \frac{3^3}{\gamma^4}\ \frac{(n-k)^{4}}{k^3(n+2s-2k)^{2}}\
 \exp\left(\frac{\gamma^{4/3}}{2}\
 \frac{k(n+2s-2k)^{2/3}}{(n-k)^{4/3}}\right)
\end{align*}
and
\begin{align*}
\sum_{\ell=\ell_0+O(\sqrt{\ell_0})}\exp(b(\ell)-b(\ell_0))
&=(1+o(1))\sum_{\ell=\ell_0+O(\sqrt{\ell_0})}
\exp\left(-\frac{3(\ell-\ell_0)^2}{4\ell_0}\right)\\
&=(1+o(1))\sqrt{\frac{4\pi \ell_0}{{3}}}\\
&=(1+o(1))\frac{2\pi^{1/2}\gamma^{2/3}}{3}\
\frac{k^{1/2}(n+2s-2k)^{1/3}}{(n-k)^{2/3}}.
\end{align*}
This implies
\begin{align*}
\sum_{\ell\ge 1}\left(\frac{\phi}{\ell^{3/2}}\right)^\ell \ell^{-3}
=& (1+o(1))\ \ell_0^{-3}\ \exp(b(\ell_0))
 \sum_{\ell=\ell_0+O(\sqrt{\ell_0})}\exp(b(\ell)-b(\ell_0))\\
=& (1+o(1))\frac{2\sqrt{\pi}3^2}{\gamma^{10/3}}
\frac{(n-k)^{10/3}}{k^{5/2}(n+2s-2k)^{5/3}}\\
&\quad\quad\quad\times \exp\left(\frac{\gamma^{4/3}k}{2}
\left(\frac{(n+2s-2k)}{(n-k)^2}\right)^{2/3}\right),
\end{align*}
and hence (\ref{eq:firstrange1}) becomes
\begin{align}
\pl(n,n/2+s)
&=(1+O(1/n))
\frac{g\ 3^{5/2}}{2^{7/2} {\pi}^{1/2}
e^{3/4}\gamma^{10/3}}n^{n+7/6}e^{n/2+s}
\nonumber\\
&\hspace{-10ex}\times \sum_k (1+O(1/k)+O(k/n))
\rho(n-k,n/2+s-k)\label{eq:firstrange2}\\
&\hspace{-10ex}\times
\frac{(n-k)^{2s-k}}{(n+2s-2k)^{n/2+s-k}}\frac{1}{k^{7/2}e^k}
\exp\left(\frac{\gamma^{4/3}k}{2}
\left(\frac{(n+2s-2k)}{(n-k)^2}\right)^{2/3}\right).\nonumber
\end{align}

We shall sum over $k$ in (\ref{eq:firstrange2}),
or, more specifically, over $r$ for $k=2s+r$, where $r=r(n,s)$
will shortly be determined.
Letting $k=2s+r$, we estimate the summands in (\ref{eq:firstrange2}) as
\begin{align*}
\frac{(n-k)^{2s-k}}{(n+2s-2k)^{n/2+s-k}}
&=\,  (n-2s)^{-(n/2-s)}
\left(1-\frac{r}{n-2s}\right)^{-r}\left(1-\frac{2r}{n-2s}\right)^{r-n/2-s}
\nonumber\\
&\stacksign{(\ref{eq:1+x})}{=}\, (n-2s)^{-(n/2-s)}
\exp\left(r-\frac{r^3}{6(n-2s)^2}+O\left(\frac{r^4}{(n-2s)^3}\right)\right).
\nonumber
\end{align*}
Therefore (\ref{eq:firstrange2}) becomes
\begin{equation}\label{eq:firstrange-sumr}
\begin{aligned}
\pl(n,&n/2+s)
=(1+O(1/n)+O(1/s)+O(s/n))\frac{g\ 3^{5/2}}{2^{7} {\pi}^{1/2}
e^{3/4}\gamma^{10/3}}n^{n+7/6}e^{n/2-s}\\
&\times  (n-2s)^{-(n/2-s)}
\exp\left(\frac{\gamma^{4/3}s}{(n-2s)^{2/3}}\right)
\sum_r \rho(n-2s-r,n/2-s-r)\varphi(r),
\end{aligned}
\end{equation}
where the function $\varphi(r)=\varphi_{n,s}(r)$ is defined as
\begin{equation*}
\varphi(r):=
\left(s+r/2\right)^{-7/2}\exp\left(\frac{\gamma^{4/3}r}{2(n-2s)^{2/3}}
-\frac{r^3}{6(n-2s)^2}+O\left(\frac{r^4}{(n-2s)^3}\right)
+O\left(\frac{r}{n}\right)\right).
\end{equation*}

Observe that the main contribution to  $\varphi(r)$,
and therefore to (\ref{eq:firstrange-sumr}),
comes from the terms $r=O(n^{2/3})$. Since
$n/2-s-r=(n-2s-r)/2-r/2$, from Lemma~\ref{thm:Britikov} (ii)
and Definition~\ref{def:nu}
we have $\rho(n-2s-r,n/2-s-r)\to\nu(-x/2)$,
when  $ r /(n-2s-r)^{2/3}  \to x$.
Thus, the sum over $r=O(n^{2/3})$ in (\ref{eq:firstrange-sumr})
can be replaced by an integral over $x=rn^{-2/3}$, and we get
\begin{align*}
\sum_{r} &\rho(n-2s-r,n/2-s-r)\varphi(r)\\
&=(1+o(1))\ n^{2/3}\ s^{-7/2}\int_{-\infty}^\infty
\exp\left(-\frac{x^3}{6}+\frac{\gamma^{4/3}x}{2}\right)
\nu\left(-\frac x2\right)dx.
\end{align*}
As as consequence, the first part of the theorem follows.

For the second part of the assertion note that in this case
the main contribution to the sum (\ref{eq:firstrange1}) follows from
$\ell$'s close to $\ell_0$ given in (\ref{eq:l}) and $k=(2+o(1))s$, and
so \aas $\ex(n,M)= (\frac{2\gamma^{4/3}}{3}+o(1))\frac{s}{n^{2/3}}$.
Therefore, by Theorem~\ref{thm:complex}, \aas $P(n,M)$ is clean,
its kernel has
$$(2+o(1))\ex(n,M)=\Big(\frac{4\gamma^{4/3}}{3}+o(1)\Big)\frac{s}{n^{2/3}}$$
vertices, and $\Pcr(n,M)=(2\gamma^{2/3}+o(1)) \frac{s}{n^{1/3}}$.
\end{proof}

\subsection{Middle range}\label{sec:middlerange}
As far as $M/n$ is bounded away from both $0$ and $1$ we can prove results
similar to Theorem~\ref{thm:super} but,
since now $\ell=O(n^{1/3})=O(k^{1/3})$,
we can estimate $\pl(n,M)$ only up to a constant factor.

\begin{theorem}\label{thm:middle}
If $M=an$ for a constant $1/2<a<1$, then
\begin{equation*}
\pl(n,M)=\Theta(1)
n^{an-5/3}\left(\frac{e}{2-2a}\right)^{n-an}
\exp\left(\gamma^{4/3}(a-1/2)n^{1/3}\right).
\end{equation*}
Furthermore, \aas $L_1(n,M)=(2a-1+o(1))n$, $\ex(n,M)=\Theta(n^{1/3})$,
and $\Pcr(n,M)=\Theta(n^{2/3})$.
\end{theorem}
\begin{proof}
Following the lines of the proof of Theorem~\ref{thm:super},
but with $s$ replaced by $an-n/2$, yields the assertion.
\end{proof}

Using Theorems~\ref{thm:subcritical}-\ref{thm:middle} one can find
that the threshold for the property that $P(n,M)$ has
the chromatic number four is $M=n+o(n)$.

\begin{theorem}\label{cor:mid}
Let $\eps>0$.
\begin{enumerate}
\item If $M\le (1-\eps)n$, then \aas $\chi(P(n,M))=3$.
\item If $M\ge (1+\eps)n$, then \aas $P(n,M)$ contains a
copy of $K_4$, and as a consequence, $\chi(P(n,M))=4$.
\end{enumerate}
\end{theorem}

\begin{proof} Here we only sketch the argument. Let $M=an$, $1/2<a<1$.
Then \aas the kernel of $P(n,M)$ has deficiency $\Theta(1)$
and $\Theta(n^{1/3})$ vertices
(see Theorem~~\ref{thm:complex} and~\ref{thm:middle}).
Furthermore, on the edges of the kernel we need to place
$\Theta(n^{2/3})$ vertices of the core. Thus, the probability
that on some edge we place fewer than five vertices is $\Theta(n^{-1/3})$,
and so there are \aas at most $\ln\ln n$ edges of the kernel
of  $P(n,M)$ which contain fewer than three vertices of the core.
Moreover,  none of the vertices of the kernel is incident with more than
one such  edge.
It is easy to see that such a graph can be colored using
three colors. A similar argument shows that  $\chi(P(n,M))\le 3$ for
$M\le n/2+o(n)$. On the other hand, from the formula for $\pl(n,M)$
for $M\ge (1+\eps)n$ by Gim{\'e}nez and Noy~\cite{GN} and
Chebyshev's inequality it follows
that for such an $M$ the graph $P(n,M)$ contains a copy of $K_4$
(in fact it \aas contains $\Theta(n)$ copies of $K_4$, in
which three vertices have  degree three in $P(n,M)$).
\end{proof}

Let us remark that Dowden~\cite{Dowden} studied the probability
that $P(n,M)$ contains a given subgraph and determined its asymptotic
behavior depending on the ratio $M/n$. In particular,
Theorem~\ref{cor:mid} (ii) is relevant to his Theorem 17.

\subsection{Second critical range}\label{sec:secondrange}
In the previous section we showed that as far as
$M=an$, and $a\in (1/2,1)$, the size of the largest component
grows with $M$, but its density does not depend much on the value
of $a$ and in the whole range is of the order $n^{1/3}$.
Clearly, this situation must change when the size of
the largest component is $n-o(n)$. Indeed, starting from some
point, the increase in the number of edges of $P(n,M)$
must contribute to the density of the largest component,
since  when $M=n+t$, for $t$ large enough,
we should expect $\ex(n,M)=(1+o(1))t$.
Our next result states that this change occurs when $M=n+O(n^{3/5})$.

\begin{theorem}\label{thm:secondrange}
Set $M=n+t$, where $t=o(n)$.
\begin{enumerate}
\item[(i)] Let $w=w(n,t)
=\frac{\gamma^{4/3}(n-2|t|)}{3\cdot 2^{2/3}|t|^{2/3}}.$
If $t\ll -n^{3/5}$, but $n/2+t\gg n^{2/3}$, then
\begin{equation}\label{eq:sec1}
\begin{aligned}
\pl(n,M)
=\Theta(1)\ &n^{n-1/2} \frac{(2|t|+2w)^{t+1/6}}{(3|t|+5w)^{1/2}w^{5/2}}
\left(\frac{|t|}{|t|+w}\right)^{w}\\
&\times \exp\left(\frac{5w}{2}+|t|-\frac{3w^2}{n-2|t|}
+\beta\cdot \frac{\gamma^2}{3^{3/2}2} \frac{n-2|t|}{|t|}\right).
\end{aligned}
\end{equation}
Furthermore, \aas $L_1(n,M)=n-(2+o(1))|t|$,
$\ex_{\rm c}(n,M)=(1+o(1))w$, and
$\Pcr_{\rm c}(n,M)
=(\frac{\gamma^{2/3}}{2^{1/3}}+o(1))\frac{n-2|t|}{|t|^{1/3}}$.
\item[(ii)] Let $b$ be the unique positive solution of the equation
$b^{3/2}(b-c)=\frac{\gamma^{2}}{2\cdot 3^{3/2}}$.
If $t=c\cdot n^{3/5}$ for $c\in (-\infty, \infty)$, then
\begin{equation}\label{eq:sec:add}
\begin{aligned}
\pl(n,M)
=\Theta(1)\ &n^{n-1/2}t^{-17/6} (2(b/c-1)t)^{t}\\
&\times \exp\left(\left(\frac{5b}{2c}-1\right)t
-\frac{3b}{c}\left(\frac{b}{c}-1\right)\frac{t^2}{n}
+\beta\cdot \frac{(bt/c)^{3/2}}{n^{1/2}}\right).
\end{aligned}
\end{equation}
Furthermore, \aas $L_1(n,M)=n-(2b-2c+o(1))n^{3/5}$,
$\ex_{\rm c}(n,M)=(b+o(1))n^{3/5}$, and $\Pcr_{\rm c}(n,M)
=\Theta(n^{4/5})$. %
\item[(iii)] Let $z=z(n,t)
=\frac{\gamma^{2}}{2\cdot 3^{3/2}}\left(\frac{n}{t}\right)^{3/2}$.
If $t\gg n^{3/5}$, then
\begin{equation}\label{eq:sec2}
\begin{aligned}
\pl(n,M)=\Theta(1)\
&\frac{n^{n-1/2}\ (2z)^{t +1/6}}{(2t+5z)^{1/2}
{(t+z)^{5/2}}} {\left(\frac{t}{t+z}\right)^{3(t+z)/2}}\\
&\times
\exp\left(\frac{3}{2}t+\frac{5}{2}z-\frac{3z(t+z)}{n}
+\beta\cdot \frac{t^{3/2}}{n^{1/2}}\right).
\end{aligned}
\end{equation}
Furthermore, if in addition $t\ll n^{2/3}$, \aas
$L_1(n,M)=n-(\gamma^2/3^{3/2}+o(1))({n}/{t})^{3/2}$,
$\ex_{\rm c}(n,M)=(1+o(1)(t+z)$, and $\Pcr_{\rm c}(n,M)
=\Theta(\sqrt{nt})$.
\end{enumerate}
\end{theorem}

\begin{proof}
In order to show (i) we follow the argument presented in
Section~\ref{sec:supercritical}.
Thus, from Theorem \ref{thm:complex} and Lemma \ref{thm:Britikov}, we get
\begin{align}
\pl(n,n+t)
=&\sum_{k,\ell}\binom{n}{k}C(n-k,n-k+\ell)U(k,k+t-\ell)  \nonumber\\
=&(1+o(1))\frac{g\ 3^{1/2}}{{\pi}  2^{5}}\ {n^{n-1/2}}
\sum_{\ell} \left(\frac{\gamma^{2}e^{3/2}}{3^{3/2}}\right)^{\ell}
  \frac{1}{\ell^{3\ell/2+3}}
\sum_k \psi(k),\label{eq:secondrange1}
\end{align}
where the function $\psi(k)=\psi_{n,t}(k)$ is defined as
\begin{align*}
\psi(k)&:=
\rho(k,k+t-\ell) (k(k+t-\ell))^{-1/2} (n-k)^{3\ell/2} \ k^{k+2t-2\ell}\
\left(\frac{e}{2(k+t-\ell)}\right)^{k+t-\ell}\nonumber\\
&\hspace{-5ex}\times
\exp\left(-\frac{k+t-\ell}{k}-\left(\frac{k+t-\ell}{k}\right)^2
+O\left(\frac{1}{k}\right)+O\left(\frac{k}{n}\right)
+\beta\sqrt{\frac{\ell ^3}{n-k}}+O\left(\frac{\ell^2}{n-k}\right)\right).
\end{align*}

In order to estimate (\ref{eq:secondrange1}) observe that the maximum
of the sum $\sum_k \psi(k)$ is taken at $k_0=2(\ell-t)$; more precisely,
the main contribution to the sum comes from the terms
$k=2(\ell-t)+r$, where $r=O\big((2(\ell-t))^{2/3}\big)$.
Thus,
\begin{align*}
\sum_k \psi(k)
=(1&+o(1)) e^{-3/4} 2^{1/2}
\frac{e^{\ell-t}(n-2(\ell-t))^{3\ell/2}}{(2(\ell-t))^{\ell-t+1}}\\
&\times
\sum_{r} \rho(2(\ell-t)+r,\ell-t+r)
\exp\left(\frac{r^3}{6(2(\ell-t))^2}-\frac{3r\ell}{2(n-2(\ell-t))}\right)\\
&\times \exp\left(\beta\sqrt{\frac{\ell^3}{n-2(\ell-t)}}
+O\left(\frac{\ell^2}{n-2(\ell-t)}\right)\right),
\end{align*}
where we sum  over $r=O\big((2(\ell-t))^{2/3}\big)$.
Note also that $\ell-t+r=\frac{2(\ell-t)+r}{2}+\frac{r}{2}$, and so
if $ r/(2(\ell-t)+r)^{2/3}\to x$,  then
$\rho(2(\ell-t)+r,\ell-t+r)\to\nu\left(\frac{x}{2}\right)$.

Next note that the main contribution to the sum over $\ell$ in
(\ref{eq:secondrange1}) comes
from the terms $\ell=\ell_0+O(\sqrt{\ell_0})$,  where
$\ell_0=\frac{\gamma^{4/3}(n-2(\ell_0-t))}{3 (2(\ell_0-t))^{2/3}}$,
so (\ref{eq:secondrange1}) can be estimated by
\begin{equation}\label{eq:main2}
\begin{aligned}
\pl(n,n+t)
=(1&+o(1))\frac{g\ 3^{1/2}}{{\pi}  2^{9/2} e^{3/4}}\ {n^{n-1/2}}\
\int_{-\infty}^{\infty} \exp\left(\frac{x^3}{6}
-\frac{\gamma^{4/3}x}{2}\right) \nu\left(\frac{x}{2}\right)dx\\
&\times
A(n,t)\ \exp\left(\beta\sqrt{\frac{\ell_0^3}{n-2(\ell_0-t)}}
+O\left(\frac{\ell_0^2}{n-2(\ell_0-t)}\right)\right),
\end{aligned}
\end{equation}
where the function $A(n,t)$ is defined as
\begin{equation*}
A(n,t):=\sum_{\ell} \left(\frac{\gamma^{2}e^{3/2}}{3^{3/2}}\right)^{\ell}\
\frac{e^{\ell-t}(n-2(\ell-t))^{3\ell/2}}{\ell^{3\ell/2+3}(2(\ell-t))^{\ell-t+1/3}}.
\end{equation*}

The behavior of $A(n,t)$ depends on the range of $t$.
Thus, we consider three cases, corresponding to the periods described
in the three parts of the assertion of Theorem~\ref{thm:secondrange}.

\medskip

First, let $t\ll -n^{3/5}$, but $n/2+t\gg n^{2/3}$.
The main contribution to the sum $A(n,t)$ comes from
the terms $\ell=\ell_0+O(\sqrt{\ell_0})$ with $\ell_0=w$,
where $$w=w(n,t):=\frac{\gamma^{4/3}(n-2|t|)}{3\cdot |2t|^{2/3}}.$$
In this case we have
\begin{equation*}
A(n,t)
=(1+o(1))(2\pi)^{1/2}\frac{(2(w-t))^{t+1/6}}{(5w-3t)^{1/2}w^{5/2}}
\left(\frac{|t|}{w-t}\right)^{w}
\exp\left(\frac{5w}{2}-t-\frac{3w^2}{n+2t}\right)\,,
\end{equation*}
and (\ref{eq:sec1}) follows. In order to get the
information on the structure of $P(n,M)$ for this range
one should apply Theorem~\ref{thm:complex} and observe that
\begin{equation*}
\sqrt{\frac{\ell_0^3}{n-2(\ell_0-t)}}
= (1+o(1))\frac{\gamma^2}{3^{3/2}2} \frac{n+2t}{|t|}.
\end{equation*}

\medskip

Now let $t=c\, n^{3/5}$ for some constant $c\in (-\infty, \infty)$.
In this case, the main contribution to the sum $A(n,t)$ comes from
the terms $\ell=\ell_0+O(\sqrt{\ell_0}) $ with
$\ell_0=b\, t/c$, where $b$ is the unique positive solution of the equation
$b^{3/2}(b-c)=\frac{\gamma^{2}}{2\cdot 3^{3/2}}$. We have
\begin{align*}
A(n,t)
=(1+o(1))(2\pi)^{1/2}&\frac{(2(b-c))^{1/6}c^{3}}{(5b-3c)^{1/2}b^{5/2}}
\ t^{-17/6} (2(b/c-1)t)^{t}\\
&\times\exp\left(\left(\frac{5b}{2c}-1\right)t
-\frac{3b}{c}\left(\frac{b}{c}-1\right)\frac{t^2}{n}\right),
\end{align*}
and
\begin{equation*}
\sqrt{\frac{\ell_0^3}{n-2(\ell_0-t)}}
=(1+o(1)) \frac{(bt/c)^{3/2}}{n^{1/2}}.
\end{equation*}
Therefore, (ii) follows from (\ref{eq:main2})
and Theorem~\ref{thm:complex}.

\medskip

Finally, let $t\gg n^{3/5}$, but $t=o(n)$.
Then the main contribution to the sum $A(n,t)$ comes from
the terms $\ell=\ell_0+O(\sqrt{\ell_0})$ with $\ell_0=t+z$,  where
$$z=z(n,t):=\frac{\gamma^{2}}{2\cdot 3^{3/2}}
\left(\frac{n}{t}\right)^{3/2}.$$
We have
\begin{align*}
A(n,t)
=(1+o(1))(2\pi)^{1/2}&\frac{(2z)^{t +1/6}}{(2t+5z)^{1/2} {(t+z)^{5/2}} }
{\left(\frac{t}{t+z}\right)^{3(t+z)/2}}\\
&\times \exp\left(\frac32t+\frac52z-\frac{3z(t+z)}{n}\right),
\end{align*}
which, together with (\ref{eq:main2}), gives (\ref{eq:sec2}).

The second part of (iii) follows from (\ref{eq:main2}),
Theorem~\ref{thm:complex}, and the observation that
\begin{equation*}
\sqrt{\frac{\ell_0^3}{n-2(\ell_0-t)}}
=(1+o(1)) \frac{t^{3/2}}{n^{1/2}}.
\end{equation*}
We also remark that we expect  the estimates for
$\ex_{\rm c}(n,M)$ and $\Pcr_{\rm c}(n,M)$
to hold also for  $n^{2/3}\le t\ll n$  but the error term
$O(t^{3/2}n^{-1/2})$  in (\ref{eq:sec2}) becomes too large
to provide a precise information on the structure
of $P(n,M)$. Let us also point out that if Giant Conjecture is true,
then $\ex(n,M)=(1+o(1))\ex_{\rm c}(n,M)$
and $\Pcr(n,M)=(1+o(1))\Pcr_{\rm c}(n,M)$
in the whole range of $M$.
\end{proof}

\section{Concluding remarks}

Let us first propose some heuristic which explains the
behavior of $P(n,M)$ as described in
Theorems~\ref{thm:subcritical}--\ref{thm:secondrange}.
Note that the number $\pl(n,M)$ of planar graphs with
$n$ vertices and $M$ edges can be computed in two different ways.
Our estimates were based on the formula
\begin{equation*}
\pl(n,M) = \sum_{k,\ell}\binom{n}{k}C(k,k+\ell)U(n-k,M-k-\ell),
\end{equation*}
where we extracted from the graph its complex part. But one could also
use the formula
\begin{equation}\label{eq:main3}
\pl(n,M) = \sum_{k,\ell}\binom{n}{k}C^{conn}(k,k+\ell)R_{pl}(n-k,M-k-\ell),
\end{equation}
where we first identify in the graph the largest component of
$k$ vertices and $k+\ell$ edges, which typically
is complex and unique, and then we supplement it by a random planar  graph
of $n-k$ vertices and $M-k-\ell$ edges.
However, Theorem~\ref{thm:complex}(iv) states that
(at least for small $\ell$)
a graph chosen at random from all complex planar graphs with
$k$ vertices and $k+\ell$ edges consists of the giant component
of size $k-O(k/\ell)$  and, possibly,
some small  components of finite complexity and size $\Theta(k/\ell)$.
If the Giant Conjecture is true, it is in fact the case for all
values of $\ell\le k$.   Thus,
the planar graph which is outside the largest component must contain
just a few (if any) components which are complex. It happens only if
its density is such as the density of the standard uniform
graph model in the critical period. Consequently, in (\ref{eq:main3}),
we must have
\begin{equation}\label{eq:heur1}
M-k-\ell= (n-k)/2+\Theta((n-k)^{2/3})\,,
\end{equation}
and, since in the critical period the sizes of all complex components
of a random graph on $n$ vertices are of the order $n^{2/3}$,
\begin{equation}\label{eq:heur2}
k/\ell=\Theta((n-k)^{2/3})\,.
\end{equation}

Let us check what it gives when $M=n/2+s$ with $s\gg n^{2/3}$
to ensure that $P(n,M)$ contains a complex and unique giant component.
If $k\ll M$, then from (\ref{eq:heur1}) we get
$$L_1(n,M)=2M-n+O(\ell)+O(n^{2/3})=2s+O(\ell)+O(n^{2/3})\,, $$
and, by (\ref{eq:heur2}),
$$\ell=\Theta\Big(\frac{k}{(n-k)^{2/3}}\Big)
=\Theta\Big(\frac{s}{(n-s)^{2/3}}\Big)\,$$
which fits perfectly the estimates from Theorem~\ref{thm:super}
and~\ref{thm:middle}.
Now let $M=n-t$, $t=o(n)$. Then, as before, from~(\ref{eq:heur1})
and~(\ref{eq:heur2}) we infer that
\begin{equation}\label{eq:heur3}
n-L_1(n,M)=2t+O( (  n-L_1(n,M) )^{2/3} )+O(\ell),
\end{equation}
and
\begin{equation}\label{eq:heur4}
\ell= \Theta\Big(\frac{n}{(n-L_1(n,M))^{2/3}}\Big).
\end{equation}
Note that when $-t$ is large, then, clearly, $\ell=\Theta(|t|)$,
and the error term $O(\ell)$ in (\ref{eq:heur3}) is equal to the
main term $2t$. In order to compute this `threshold value' of
$t$ one should use (\ref{eq:heur4}) and substitute to it $\ell=\Theta(|t|)$
and $n-L_1(n,M)=\Theta(t)$. It gives $|t|=\Theta(n^{3/5})$
as the threshold value for the property that $\ell=\Theta (|t|)$.
Consequently, if $t\gg n^{3/5}$, then
we have $n-L_1(n,M)=(2+o(1))t$, as stated
in Theorem~\ref{thm:secondrange}(i), while
for $t\ll -n^{3/5}$, we have $\ell=\Theta(|t|)$ and,
by (\ref{eq:heur4}), $n-L_1(n,M)=\Theta((n/t)^{3/2})$,
which agrees with Theorem~\ref{thm:secondrange}(iii).

Since for $M<an$, where $a<1$, the random graph $P(n,M)$ is \aas
quite sparse (see Theorem~\ref{thm:middle}), while for $M>an$, $a>1$,
it is quite dense (e.g.\ \aas it contains a copy of any given planar graph),
it seems that the most interesting period in the evolution of $P(n,M)$
is for $M=n-o(n)$. This intuition is confirmed by Theorem~\ref{cor:mid},
but clearly a lot remains to be done.
For instance, it seems that the correct threshold function for the
property that $\chi(P(n,M))=4$ is $M=n+\Theta(n^{7/9})$, more precisely,
we conjecture that the following holds.

\medskip

{\bf Conjecture} {\sl If $(M-n)/n^{7/9}\to 0$,
then \aas $\chi(P(n,M))\le 3$,
while if  $(M-n)/n^{7/9}\to \infty$, then \aas $P(n,M)\supseteq K_4$ and
so $\chi(P(n,M))=4$.}

\medskip

Let us briefly justify the above claim that $K_4$ emerges in $P(n,M)$
when $M=n+\Theta(n^{7/9})$.
First of all, in order to have a single copy of $K_4$ in
$P(n,M)$ we need a lot of copies of
$K_4$ in the kernel of
the largest component of $P(n,M)$. The number of copies of $K_4$
in the kernel should be of the same order as the number of vertices
of degree four in the kernel, which, in turn, is expected to be
of the order of deficiency of the graph. Once we have a copy of $K_4$
in the kernel, the probability that after placing vertices of the core
at the edges of the kernel none of these vertices will be put at one
of six edges of $K_4$ is $\Theta((\ker(n,M)/\core(n,M))^6)$.  Hence,
if our estimates for $\ker(n,M)$ and $\core(n,M)$ from
Theorem~\ref{thm:secondrange}(iii) remain valid for all values of
$M$ such that $n^{3/5}\ll M-n\ll n$, then the expected number of
copies of $K_4$'s in $P(n,M)$ is of the order
$$\dd(n,M) \cdot\Big(\frac{\ker(n,M)}{\core(n,M)} \Big)^6\sim
\sqrt{\frac{\ell^3}{k}}\Big(\frac{\ell}{\sqrt{k\ell}} \Big)^6=
\ell^{9/2}k^{-7/2}\sim t^{9/2}n^{-7/2}.$$
Thus, this number is bounded away from zero for $t=\Theta(n^{7/9})$.

Let us also add a few words  on the models of random planar graphs
different from $P(n,M)$. One of the most natural one is the
graph obtained by the random planar process, when we
add to  an empty graph on
$n$ vertices $M$ edges one by one  each time
choosing a new edge uniformly at random from
all pairs which preserve  planarity of the
graph (see Gerke {\it et al.}~\cite{GSST}). In this model
the structure of components is similar to that of a standard
graph $G(n,M')$ for an appropriately chosen $M'\ge M$. Another
model of random planar graph is the binomial random graph $P(n,p)$,
when we look at properties of $G(n,p)$ conditioned on the fact that
it is planar. Equivalently, one can view $P(n,p)$ as the graph
chosen from the family ${\mathcal P}(n)$ of all planar graphs
on $n$ vertices in such a way that each $G\in {\mathcal P}(n)$ appears
as $P(n,p)$ with the probability
\begin{equation}\label{eq:c:1}
\pr(P(n,p)=G)=p^{e(G)}(1-p)^{\binom n2 -e(G)}/Z(n,p),
\end{equation}
where $e(G)$ denotes the number of edges of $G$, and
\begin{equation}\label{eq:c:2}
Z(n,p)=\sum_{G\in \mathcal P (n)}p^{e(G)}(1-p)^{\binom n2 -e(G)}\,.
\end{equation}
Since clearly for every property $\mathcal A$
$$\pr(P(n,p)\ \textrm{has}\ \mathcal A|e(P(n,p))=M)=
\pr(P(n,M)\ \textrm{has}\ \mathcal A)\,,$$
once we determine  the typical number of edges in $P(n,p)$
the problem of finding properties of $P(n,p)$ reduces
to studying these properties for $P(n,M)$. From the estimates
of $\pl(n,M)$ given in Theorems~\ref{thm:subcritical}--\ref{thm:middle}
and \ref{thm:secondrange} it follows
that if $np\le 1$ then \aas $P(n,p)$ has $M=(1+o(1))p\binom n2$ edges;
if $1/n\le p\ll n^{-3/5}$, then \aas $M=n-(1+o(1))/(2p)$;
for $p=O(n^{-3/5})$ we are in the second critical
period, i.e.\ \aas $M=n+O(n^{3/5})$; finally for
$p\gg n^{-3/5}$ we have \aas $M=n+\Theta(p^{2/3}n)$.
Note that a large part of the evolution of $P(n,p)$,
when $1/n\ll p\ll 1$, corresponds to the period
of evolution of $P(n,M)$ when $M=n+o(n)$ which, as we have
already remarked, is crucial for many properties of $P(n,M)$.

Another interesting model is a random cluster model $P(n,M,q)$
on planar graphs when for every labeled planar graph
with vertex set $[n]$ we put
\begin{equation*}\label{eq1}
\pr(P(n,M,q)=G)=q^{c(G)}/Z(n,M,q),
\end{equation*}
where $q>1$ is a parameter,  $c(G)$ stand for the
number of components in $G$, and $Z(n,M,q)$ is the normalizing factor.
In a similar way one can define $P(n,p,q)$ adding factors $q^{c(G)}$
to the right hand sides of~(\ref{eq:c:1}) and~(\ref{eq:c:2}).
It is well known that the additional cluster factor $q^{c(G)}$ in,
say, the standard model $G(n,p)$ leads to an interesting phenomena
such as the discontinuous
phase transition which occurs in $G(n,p,q)$ for $q>2$ (cf.\ Luczak and
{\L}uczak~\cite{LL}). Unfortunately, no such event can be observed
in the planar case. The evolution of $P(n,M,q)$ is quite similar to that
of $P(n,M)=P(n,M,1)$. The reason is quite simple:
the giant complex component
of $P(n,M)$ is very sparse until it reaches the size $n-o(n)$ and so
the number of components is always close to  $n-M+o(n)$ and, as calculations
show, cannot be influenced much by the presence
of the additional factor $q^{c(G)}$.
The asymptotic behavior of  $P(n,p,q)$ does not depend very much
on the value of $q$ either  except for the scaling:
the number of edges of $P(n,p,q)$ is roughly the same as for $P(n,p/q)$.

\section*{Acknowledgments}
A part of this work was carried out when both the authors visited
the Mittag-Leffler Institute and completed during the first author's
stay at Adam Mickiewicz University within the DFG Heisenberg Programme.
We would like to thank these institutions for their support.
The second author is supported by the Foundation for Polish Science.

\end{document}